\documentclass[journal]{IEEEtran}
\usepackage[noadjust]{cite}
\usepackage{graphicx}
\usepackage{algorithmic}
\usepackage{amsmath,amssymb,amsthm}
\usepackage{upgreek}
\usepackage{color}

\newcommand{\dft}{\xi}
\newcommand{\cu}{\mathrm i}
\newcommand{\comment}[1]{}
\allowdisplaybreaks

\begin{document}
\DeclareGraphicsExtensions{.pdf}
\graphicspath{{./}}
\theoremstyle{definition}
\newtheorem{theorem}{Theorem}
\newtheorem{lemma}{Lemma}
\newtheorem{remark}{Remark}
\newtheorem{definition}{Definition}

\title{Binary Discrete Fourier Transform and its Inversion}

\author{Howard~W.~Levinson and Vadim~A.~Markel%
\thanks{H.W.~Levinson is with Department of Mathematics and Computer Science, Santa Clara University, Santa Clara, CA, USA (e-mail: hlevinson@scu.edu)} 
\thanks{V.A.~Markel is with Department of Radiology, University of Pennsylvania, Philadelphia, PA, USA (e-mail: vmarkel@upenn.edu)}
}

\maketitle

\begin{abstract}
A binary vector has elements that are either $0$ or $1$. We investigate whether and how a binary vector of known length $N$ can be reconstructed from a limited set of its discrete Fourier transform (DFT) coefficients. {\em A priori} information that the vector is binary provides a powerful constraint. We prove that a binary vector is uniquely defined by its first two DFT coefficients (zeroth, which gives the popcount, and first) if $N$ is prime. If $N$ has two prime factors, additional DFT coefficients must be included in the data set to guarantee uniqueness, and we find the number of required coefficients in this case theoretically. In the presence of noise, one may need to know even more DFT coefficients to guarantee stable and unique inverse solution. We say that inversion is stable if the solution does not depend on the input strongly, so that a small change in the data (i.e., due to noise or loss of precision) does not result in a dramatically different inverse solution. Since the inverse solutions that we seek are discrete, stability in the context of this paper implies that small changes in the data do not change the inverse solution at all. Our results indicate that stable inversion can be obtained when the number of known coefficients is about $1/3$ of the total. This entails the effect of super-resolution (the resolution limit is improved by the factor of $\sim 3$). Two algorithms for solving the inverse problem numerically are proposed and tested. The first algorithm is combinatorial and suitable for problems with $N \lesssim 60$. Although the problem in general is NP-hard, the required computational complexity in a typical application of the algorithm is much smaller than that of exhaustive search. The second algorithm is based on optimization of a non-convex continuous functional with random jumps out of local minima, and its computational complexity is algebraic, although the number of required iterations is difficult to estimate. This latter method is applicable to much larger values of $N$, as we demonstrate numerically for $N=199$.
\end{abstract}


\section{Introduction}
\label{sec:intro}

Super-resolution in imaging and signal processing is a subject of significant current importance. Fundamentally, the resolution limit is related to experimental inaccessibility of the spatial Fourier harmonics of an object beyond some band limit~\cite{maznev_2017_1}. The resulting image is therefore a low-path filtered version of the object. The classical Abbe limit on the resolution of optical systems is of this nature~\cite[Sec.~13.1.2]{born_book_99}. There exist many more examples of imaging or signal reconstruction problems in which Fourier components of a signal beyond some band limit are lost, corrupted or inaccessible. 

One of the most powerful approaches for achieving super-resolution, as well as for image denoising, is based on utilization of prior information~\cite{nasrollahi_2014_1, romano_2017_1}. The latter can be introduced explicitly as a probability density in the Bayesian framework~\cite{watzenig_2007_1}, or implicitly as a regularizing term in a cost function~\cite{engl_2005_1}. In addition to the classical Tikhonov regularization (ridge regression), the developed techniques include nonlinear interpolation~\cite{rajan_2001_1, jiji_2006_1}, Laplacian~\cite{lagendijk_1990_1, liu_2014_1}, total variation~\cite{rudin_1992_1, beck_2009_1}, sparsity-based methods~\cite{tropp_2007_1, blumensath_2012_1, blumensath_2013_1}, and many variants of the above.

Another promising regularization technique is based on the so-called compositional constraints or the $p$-species model. It has been used in a variety of settings including microscopy~\cite{deutsch_2015_1,pfitser_2000_1}, MRI~\cite{liang_2007_1}, diffuse optical tomography~\cite{corlu_2003_1,corlu_2005_1} and electromagnetic tomography~\cite{zhang-ting_2016_1}. The approach relies on the {\em a priori} knowledge that the sample consists of two or more known ``species'', that is, materials whose properties are known. The spatial distribution of the components is however unknown and must be found by solving an inverse problem. In the context of signal processing, the $p$-species model implies that a discrete signal can take only $p$ known values.

In this paper, we examine what is perhaps the most simple and at the same time the most fundamental mathematical question of the $p$-species model. Assume that there are only two species, i.e., the signal can take only two distinct values. How many Fourier coefficients one needs to know to reconstruct the signal precisely? More specifically, let the discrete Fourier transform (DFT) coefficients be labeled by $m$ where $m \in [-M,M]$ and $2M+1=N$ is the total number of the signal samples. We also assume that the signal can take two distinct known values, say, $a$ and $b$. Can we reconstruct the signal precisely if we know the DFT coefficients only within the band limit $m \in [-L,L]$ where $L < M$? What is the smallest value of $L$ for which reconstruction is still possible?

Below, we prove that the inverse problem of reconstructing a binary signal from the band-limited DFT data has a unique solution if $L=1$ and $N$ is prime. Moreover, we derive the minimum value of $L$ that guarantees uniqueness if $N$ has two prime factors. We also discuss stability and computational efficiency of inversion. Thus, in the case of a large prime $N$, $L=1$ is theoretically sufficient to guarantee uniqueness of the inverse solution. However, finding this solution numerically can be difficult or impractical due to high computational complexity and low noise tolerance. The difficulty can be mitigated by including additional DFT coefficients into the data set. As $L$ is increased, stability of the inverse problem is improved and computational complexity is reduced. By stability we mean here a lack of sensitivity of inverse solutions to small changes in the given DFT coefficients. Thus, if inversion is stable, we expect to recover the exact binary vector even if the DFT coefficients on input are imprecise or corrupted by noise. In Supplemental Material, we provide a computational package based on the the inversion algorithms described below, and a set of examples in which forward data are rounded off to 8 and 4 significant figures. In all cases we have considered (with both prime and non-prime $N$), reconstruction becomes computationally efficient and stable  when $L \gtrsim M/3$. In this case, the data can be further rounded off to 3 or even 2 significant figures without affecting the inverse solution. This corresponds to achieving a three-fold improvement of the resolution limit.

The related previous theoretical work includes investigation of discrete Fourier sampling and recovery with both random~\cite{tropp_2008_1} and deterministic~\cite{moitra_2015_1, bailey_2012_1} sampling. Ref.~\cite{tao_2003_1} introduced an uncertainty principle for prime order cyclic groups, which is related to our uniqueness results. Also relevant are the algebraic ideas that arise in error correcting for channel coding~\cite{ryan_2009_1, oggier_2005_1} and in the investigation of the vanishing sums of the roots of unity~\cite{mann_1965_1, conway_1976_1, lenstra_1978_1, lam_1996_1, lam_2000_1}.

On the inversion and algorithmic side, we are interested in reconstructing the signal precisely and not approximately, and the form of the signal is assumed to be general. In particular, the signal may be more complicated than one or two compact pulses. Under the circumstances, application of the level-set methods~\cite{aravkin_2018_1} is impractical. Also, binarity is related to but not identical to sparsity. A sparse vector has many of its elements equal to zero but the rest can take any (unknown) values. The binary vectors considered here can have about one half of their elements equal to 0 (in the most difficult case) but the rest have the same (known) value of 1. Therefore, we did not investigate sparsity-based algorithms. Further, the inverse problem considered by us can be rephrased as a linear integer programming problem of the form $A{\tt x} = \tilde{\upvarphi}$, where $\tilde{\upvarphi}$ (defined in Section~\ref{sec:inv} below) contains the known DFT coefficients, $A$ is the relevant Fourier matrix, and ${\tt x}$ is the unknown binary vector. As there is no further objective function (just the equality constraint), using branch-and-bound techniques do not offer significant improvements as there is no optimization quantity to bound. The combinatorial algorithm described below is, in fact, an efficient implementation of the branching method in which we sequentially search the subsets of binary vectors separated from an initial guess by a monotonously-increasing distance. Without an objective function, bounding improvements must come from analyzing the feasibility of solutions within a branch. In some cases, one can apply cutting planes methods to restrict the size of the search space~\cite{balas_2003_1}.  As we have verified, a direct application of cutting planes to our problem has a marginal effect, but devising more sophisticated definitions of cutting planes is a natural area for improvement. The problem we are considering is also closely related to the 0-1 knapsack problem and its reduction to the subset sum problem~\cite{martello_2000_1}, which is known to be NP-hard~\cite{karp_1972_1}.  There are also many refined approaches for solving linear binary integer programming problems, including dynamic programming~\cite{koiliaris_2019_1, bringmann_2017_1} and probabilistic and approximate solutions~\cite{chan_2018_1, boutsidis_2009_1, howgrave_2010_1}. As a benchmark, we have solved the test inverse problems (included as examples in the computational package) using the standard linear integer programming techniques via MATLAB's mixed-integer linear programming function {\tt intlinprog}. This function was able to recover the three model vectors defined below, but required longer running time than the codes provided in the computational package.

The rest of this paper is organized as follows. In Section~\ref{sec:bin} we introduce the binary DFT, discuss its basic properties and show that it is sufficient to consider the vectors whose elements take only two known values, $0$ and $1$. In Section~\ref{sec:num}, we adduce several numerical examples that illustrate the theoretical possibility of reconstructing a binary vector from band-limited DFT data. In Section~\ref{sec:uni}, we prove two key uniqueness theorems. In Section~\ref{sec:sta} we discuss stability and in Section~\ref{sec:inv} two algorithms for numerical inversion are introduced and tested. Section~\ref{sec:disc} contains a discussion and conclusions. Supplemental Material contains the complete data set that can be used to reproduce all figures of this paper and a computational package with a detailed user guide and a set of examples. The package implements the inversion methods of Section~\ref{sec:inv}.

Below, the following acronyms are used: ``gcd'' is ``greatest common divisor'' and the symbol ``mod'' is used to denote modulus congruence, that is, $n \equiv m \pmod k$ if $n-m$ is divisible by $k$. 

\section{Binary DFT}
\label{sec:bin}

Let ${\tt v} = (v_1,v_2 \ldots , v_N)$ be a vector of length $N>1$. For simplicity, we assume that $N$ is odd. The DFT of ${\tt v}$ is given by
\begin{subequations}
\label{DFT}
\begin{align}
\label{DFT_for}
\tilde{v}_m = \sum_{n=1}^N v_n e^{\cu \dft m n } \ ,
\end{align}
where
\begin{align}
\label{xi_def}
\dft = \frac{2\pi}{N} \ , \ \ -M \leq m \leq M \ , \ \ M=\frac{N-1}{2} \ .
\end{align}
Note that $\tilde{v}_m$ is defined for any integer $m$ and is periodic so that $\tilde{v}_{m+N} = \tilde{v}_m$. It is therefore sufficient to restrict $m$ to the interval $[-M, M]$. Assuming that we know all Fourier coefficients with indexes in this interval, we can reconstruct ${\tt v}$ by the inverse DFT according to
\begin{align}
\label{DFT_inv}
v_n = \frac{1}{N}\sum_{m=-M}^M \tilde{v}_m e^{- \cu \dft n m } \ , \ \ 1\leq n \leq N  \ .
\end{align}
\end{subequations}
We will refer to ${\tt v}$ and $\tilde{\tt v} = (\tilde{v}_{-M}, \tilde{v}_{-M+1}, \ldots \tilde{v}_M)$ as to the real-space and Fourier-space vectors.

The questions we wish to address are the following. Assume that we know only some of the Fourier coefficients $\tilde{v}_m$, namely, those with the indexes bounded as $|m| \leq L \leq M$, and also that $v_n$ can take only two distinct, {\em a priori} known values, say, $a$ and $b$. What is the smallest value of $L$ for which we can reconstruct the whole vector ${\tt v}$ uniquely from the Fourier data? For which values of $L$ the reconstruction is numerically stable? Finally, we wish to develop a computational algorithm for the reconstruction.

We can simplify the problem by writing
\begin{align}
\label{yn_xn}
v_n = a + (b-a)x_n \ ,
\end{align}
where $x_n$ can take only two values, $0$ and $1$. We will say that the vector  ${\tt x} = (x_1,x_2 \ldots , x_N)$ is {\em binary}. Substituting \eqref{yn_xn} into \eqref{DFT_for}, we obtain:
\begin{align}
\label{DFT_for_ab}
\tilde{v}_m = a N\delta_{m0} + (b-a) \tilde{x}_m \ ,
\end{align}
where $\delta_{ml}$ is the Kronecker delta-symbol and $\tilde{x}_m$ are defined in terms of $x_n$ analogously to the DFT convention \eqref{DFT_for}. 

Equation \eqref{DFT_for_ab} can be inverted to yield the relation
\begin{align}
\label{txm_tym}
\tilde{x}_m = \frac{\tilde{v}_m - a N \delta_{m0}}{b-a} \ .
\end{align}
Therefore, if $\tilde{v}_m$ is known, then $\tilde{x}_m$ is also known. We thus see that the inverse problem of finding ${\tt v}$ from $\tilde{\tt v}$ is mapped onto the problem of finding a binary vector ${\tt x}$ from its DFT $\tilde{\tt x}$.

Therefore, we are interested in reconstructing the full binary vector ${\tt x}$ from a limited set of coefficients $\tilde{x}_m$ with $|m| \leq L$. If $L=0$, the only information about ${\tt x}$ that is present in the data is the popcount (the total number of 1s in ${\tt x}$). Reconstructing ${\tt x}$ with only this information is obviously impossible. However, the popcount is an important constraint on the possible solutions. In what follows, we assume that the popcount
\begin{align}
\label{r_def}
r \equiv \tilde{x}_0 = \sum_{n=1}^N x_n
\end{align}
is known with a high degree of confidence, so that \eqref{r_def} can be viewed as a hard constraint on the possible inverse solutions. 

We denote the set of all vectors ${\tt x}$ of length $N$ containing $r$ 1s by $\Omega(N,r)$. The size of this set is 
\begin{align}
\label{S_Omega_def}
S[\Omega(N,r)] = \frac{N!}{r!(N-r)!} \ .
\end{align}
It is sufficient to consider $r$ in the range 
\begin{align}
\label{r_range}
0 < r \leq M = \frac{N-1}{2}
\end{align}
because the sets $\Omega(N,r)$ and $\Omega(N,N-r)$ can be obtained from each other by the substitution $0 \leftrightarrow 1$. Therefore, the problems with $r=q$ and $r=N-q$ are mathematically identical. In what follows, we assume that $r$ is in the range \eqref{r_range}. Note that $r=0$ is a technically possible but trivial case since $r=0$ implies that all elements of ${\tt x}$ are 0s. Therefore, we exclude this possibility in \eqref{r_range}.

The band-limited to $L$ Fourier-space distance between any two vectors ${\mathtt x},{\mathtt y} \in \Omega(N,r)$ is defined as
\begin{align}
\label{chi_def}
\chi_p({\tt x}, {\tt y}; L) = \left[ \frac{1}{L}\sum_{m=1}^L \left\vert \tilde{x}_m - \tilde{y}_m \right\vert^p\right]^\frac{1}{p} \ , \ \ p\ge 1 \ .
\end{align}
Note that the term with $m=0$ is excluded from the summation. The real-space distance between ${\tt x}, {\tt y} \in \Omega(N,r)$ is defined as
\begin{align}
\label{d_def}
d({\tt x}, {\tt y}) = \frac{1}{2}\sum_{n=1}^N |x_n - y_n| \ .
\end{align}
If the distance between two vectors in $\Omega(N,r)$ is $d$, one can be obtained from the other by $d$ pair-wise switches of 0s and 1s. The possible values of $d$ are in the interval $0 \leq d \leq r$ assuming $r$ is in the range \eqref{r_range}. If $L=M$, the invertibility of DFT implies that the statements $\chi_p({\tt x}, {\tt y}; M) = 0$, $d({\tt x}, {\tt y}) = 0$ and ${\tt x} = {\tt y}$ are equivalent. However, if $L<M$, we do not generally know whether there exist pairs of distinct vectors ${\tt x} \neq {\tt y}$ for which $\chi_p({\tt x}, {\tt y}; L) = 0$.

Clearly, if $r=1$, it is sufficient to know only one additional Fourier coefficient, say, $\tilde{x}_1$. The inverse problem is then reduced to finding the position $\nu$ where the single 1 is located. We can use the equation $\tilde{x}_1 = \exp(\cu \dft \nu)$ to find $\nu$. If $\tilde{x}_1$ is in range of the forward operator, then the above equation has a unique integer solution in the interval $[1,N]$. If $\tilde{x}_1$ is not in range, then the equation has no integer solutions. One can still find the integer $\nu$ that minimizes the error $ |\tilde{x}_1 - \exp(\cu \dft \nu)|$. The case $r=2$ is also easy to analyze. The problem becomes difficult when $r \sim N/2$ and $N\gg 1$ so that $S[\Omega(N,r)]$ is combinatorially large. The remainder of his paper is largely focused on this more difficult case.

\section{A numerical example}
\label{sec:num}

As a first step, we can investigate the problem numerically. To this end, we have considered the following two model vectors ${\tt x}_{\rm mod}$:
\begin{align*}
&(a) \ N=31, \ r=15, \ S[\Omega(N,r)] = 300,540,195 \\
&{\tt x}_{\tt mod} = (1 0 0 1 0 1 1 0 0 0 0 1 1 1 0 1 1 0 1 1 0 0 0 1 1 0 1 0 1 0 0)
\end{align*}
and
\begin{align*}
&(b) \ N=33, \ r=16, \ S[\Omega(N,r)] = 1,166,803,110 \\
&{\tt x}_{\tt mod} = (1 0 0 1 0 0 1 1 0 0 0 1 1 0 0 1 1 1 0 0 1 0 1 0 1 0 0 1 1 0 1 1 0)
\end{align*}
For these values of $N$, all elements of $\Omega(N,r)$ can be constructed explicitly on a computer. 

\begin{figure}
%FIG 1
\centering
\includegraphics[]{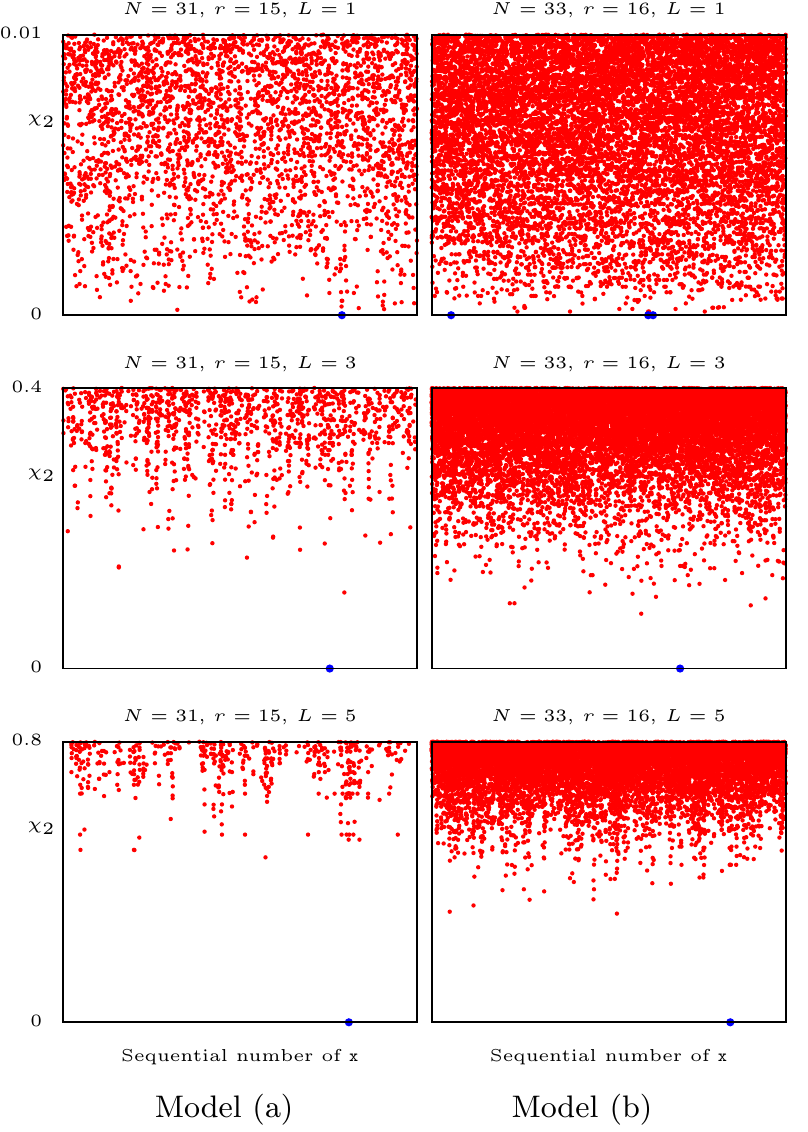}
\caption{\label{fig:1} Fourier-space distances $\chi_2({\tt x}, {\tt x}_{\tt mod}; L)$ between various ${\tt x} \in \Omega(N,r)$ and a model vector ${\tt x}_{\rm mod}$ with the same $N$ and $r$. The left and right columns correspond to the models (a) and (b), which are defined in the text. Different rows correspond to different values of $L$. All data points that fit the vertical scale of each plot are shown. Projections onto the horizontal axis are the sequential numbers of the data points and have no other significance. Due to the finite size of the dots  that are used to represent the data points, it may appear that one sequential number (a projection onto the horizontal axis) corresponds to more than one dot; in fact, this is not so. Large blue dots mark exact zeros. Note that, for $L=1$, $\chi_p$ is independent of $p$.}
\end{figure}

In Fig.~\ref{fig:1}, we display the quantities $\chi_2({\tt x}, {\tt x}_{\tt mod}; L)$ (below some thresholds) for all ${\tt x} \in \Omega(N,r)$, various $L$, and the two model vectors ${\tt x}_{\tt mod}$ defined above. It can be seen that, for $N=31$ and all $L$ considered, there exists only one ${\tt x} \in \Omega(N,r)$ for which $\chi_2({\tt x}, {\tt x}_{\tt mod}; L) = 0$. This ${\tt x}$ is the true solution, that is, it is identical to ${\tt x}_{\tt mod}$. This result implies that the knowledge of just the first two Fourier coefficients $\tilde{x}_0=r$ and $\tilde{x}_1$ suffices to find the whole ${\tt x}_{\tt mod}$. This result is surprisingly strong but, as discussed below, not unexpected. One obvious observation is that $31$ is a prime number. We will show that uniqueness of inverse solutions with $L=1$ is a general property of all prime $N$. 

In the case $N=33$ (not a prime), there are three distinct vectors ${\tt x}$ with $\chi_2({\tt x}, {\tt x}_{\tt mod};1) = 0$; only one of them is equal to ${\tt x}_{\rm mod}$. The other two are false solutions. However, uniqueness of the inverse solution is restored by selecting $L\ge 3$ ($L=2$ is still insufficient). In the next section, the theoretical reasons why $L=3$ provides the unique solution in this case will be given. 

\section{Uniqueness of inverse solutions}
\label{sec:uni}

We now fix a few definitions and prove two key uniqueness theorems.

\begin{definition}
\label{df:1}
We say that two vectors ${\tt x}, {\tt y} \in \Omega(N,r)$ are $L$-distinguishable if $\chi_p({\tt x}, {\tt y}; L) > 0$, where $1 \leq L \leq M$, and $L$-indistinguishable otherwise. If this property holds for some $p$, it holds for all $p \ge 1$, including the formal limit $p=\infty$. To generalize the definition, we say that vectors of the same length $N$ but with different popcounts $r$ are $0$-distinguishable~\footnote{Such two vectors do not belong to the same set $\Omega(N,r)$.}. All vectors in $\Omega(N,r)$ have the same popcount and are therefore $0$-indistinguishable. If two vectors are $L$-distinguishable, they are also $L'$-distinguishable for any $L'>L$. 
\end{definition}

\begin{definition} 
\label{df:2}
We use the acronym IP($N,r,L$) to denote the inverse problem of reconstructing a generic binary vector ${\tt x}$ of known length $N$ and popcount $r$ from the set of its DFT coefficients $\tilde{x}_m$ with $1 \leq m \leq L$ (the coefficient $\tilde{x}_0=r$ with $m=0$ is already included in the data set). Since $\tilde{x}_{-m} = \tilde{x}_m^*$, coefficients with negative indexes do not provide additional information and are not included in the data set. In this paper, we consider only odd $N$. The possible values of $r$ in IP($N,r,L$) are defined by the inequality \eqref{r_range}.
\end{definition}

\begin{definition}
\label{df:3}
We say that IP($N,r,L$) is uniquely solvable if all pairs of distinct vectors ${\tt x}, {\tt y} \in \Omega(N,r)$ are $L$-distinguishable. 
\end{definition}

\begin{definition} 
\label{df:4}
We say that a binary vector ${\tt x} \in \Omega(N,r)$ contains a regular polygon if there exists integers $m$ and $k$ such that $x_n=1$ for all $n \equiv m \pmod k$. We call $k$ the order of the polygon, and refer to such a polygon as a $k$-gon. Similarly, we say that ${\tt x}$ contains an empty regular polygon if $x_n=0$ for all $n \equiv m \pmod k$. Note that these definitions require $k$ to divide $N$. We say that ${\tt x}$ contains a pair of regular polygons if ${\tt x}$ contains both a regular polygon and an empty regular polygon of the same order. Regular polygons are disjoint in ${\tt x}$ if they do not share any indices $n$. 
\end{definition}

\begin{theorem}
\label{th:1}
IP($N,r,1$) is uniquely solvable for all $r$ in the interval \eqref{r_range} if $N$ is prime.
\end{theorem}

\begin{proof}
Theorem~\ref{th:1} is proved by showing that all vectors in $\Omega(N,r)$ are pairwise 1-distinguishable for $N$ prime. Consider two distinct vectors ${\tt x}, {\tt y} \in \Omega(N,r)$ and suppose that ${\tt x}$ and  ${\tt y}$ are $1$-indistinguishable. Let ${\tt z} = {\tt x} - {\tt y}$ with $z_n$ ($n=1,2,\ldots N$) denoting the real-space components of ${\tt z}$ and $\tilde{z}_m$ denoting its DFT coefficients defined according to the convention \eqref{DFT_for}. Note that ${\tt z}$ is not a binary vector as its entries can take three possible values: $0$ and $\pm 1$. Therefore, ${\tt z} \notin \Omega(N,r)$. Since we have assumed that ${\tt x}$ and ${\tt y}$ are $1$-indistinguishable, the following equality must hold:
\begin{align}
\label{z1_thm1}
0 = \tilde{z}_1 = \sum_{n=1}^N z_n e^{\cu \dft n} \ .
\end{align}
Since $N$ is prime, the $N-1$ exponential factors $e^{\cu \dft n}$ with $1 \leq n < N-1$ (excluding the term $e^{\cu \dft N}=1$) form the complete set of $N$-th primitive roots of unity~\cite{ireland_book_1990}. Therefore, the $N$-th cyclotomic polynomial~\cite{ireland_book_1990} has the form
\begin{subequations}
\begin{align}
\label{cyclotomic}
\Phi_N(x) = \prod_{n=1}^{N-1} (x - e^{\cu\dft n}) = \sum_{n=0}^{N-1}x^n \ .
% x^{N-1}+x^{N-2}+\cdots+x+1 \ ,
\end{align}
By Eisenstein's criterion, this polynomial is irreducible over the rationals~\cite{ireland_book_1990}. We also introduce the polynomial~\footnote{Note that the coefficient indexes in \eqref{g_def} are correct. An alternative way to define $g(x)$ with the same properties is $g(x)= z_1 + z_2 x + ... + z_Nx^{N-1}$.}
\begin{align}
\label{g_def}
g(x) = z_N + z_1 x + z_2 x^2 + \cdot + z_{N-1}x^{N-1} \ .
\end{align}
\end{subequations}
Assuming that \eqref{z1_thm1} holds, $g(x)$ has a root at $e^{\cu\dft}$. We now observe that both $\Phi_N(x)$ and $g(x)$ are polynomials with rational coefficients of degree $N-1$ and with the common root $e^{\cu\dft}$. Since $\Phi_N(x)$ is irreducible over the rationals, and thus a minimal polynomial, these two polynomials can only differ by a constant. This implies that $z_n = C$ for some constant $C$. Recall $z_n$ can only take values of $0$ and $\pm 1$. If $C=\pm 1$, then ${\tt x}$ is a vector of all 1s and ${\tt y}$ is a vector of all 0s (or vice versa), which violates the assumption that they are both in $\Omega(N,r)$. Hence the only possibility is $z_n=0$, which is equivalent to ${\tt x}={\tt y}$ and contradicts the initial assumption that ${\tt x}$ and ${\tt y}$ are distinct.
\end{proof}

Theorem~\ref{th:1} implies that, at least theoretically, any binary vector ${\tt x}$ of a prime length $N$ can be uniquely recovered from its two Fourier coefficients $\tilde{x}_0 = r$ and $\tilde{x}_1$. It does not imply that this problem can always be solved in a numerically stable manner. Stability of inversion is discussed in Section~\ref{sec:sta} below. 

When $N$ is not prime, we do not have such a strong statement of uniqueness. However, if $N$ has two prime factors, we can characterize the false solutions and derive the sufficient conditions of uniqueness assuming that some additional DFT coefficients are known. The following Lemma establishes the necessary and sufficient condition for two binary vectors in $\Omega(N,r)$ to be $1$-distinguishable.

\begin{lemma}
\label{lm:1}
Let $N = pq$ where $1 < p \leq q$ are primes (not necessarily distinct). Let ${\tt x} \in \Omega(N,r)$. Then ${\tt x}$ is $1$-indistinguishable from some other (different from ${\tt x}$) vector(s) in $\Omega(N,r)$ if and only if ${\tt x}$ contains at least one pair of $p$- or $q$-gons (a $p$-gon is a regular polygon with $p$ vertices) according to Definition~\ref{df:4}. Equivalently, ${\tt x}$ is $1$-distinguishable from all other vectors in $\Omega(N,r)$ if and only if it does not contain any pairs of $p$- or $q$-gons.
\end{lemma}

The proof of Lemma~\ref{lm:1} is given in Appendix~\ref{app:A}. It relies on the algebraic ideas that were used to study the vanishing sums of the roots of unity~\cite{mann_1965_1, conway_1976_1, lenstra_1978_1, lam_1996_1, lam_2000_1}. Geometrically, the proof states that the only way ${\tt x}$ and ${\tt y}$ in $\Omega(N,r)$ can be $1$-indistinguishable is if they agree element-wise except for the locations that make up an equivalent number of regular $p$- or $q$-gons. This is illustrated in Fig.~\ref{fig:2}. 

\begin{figure}
%FIG 2
\centering
\includegraphics[]{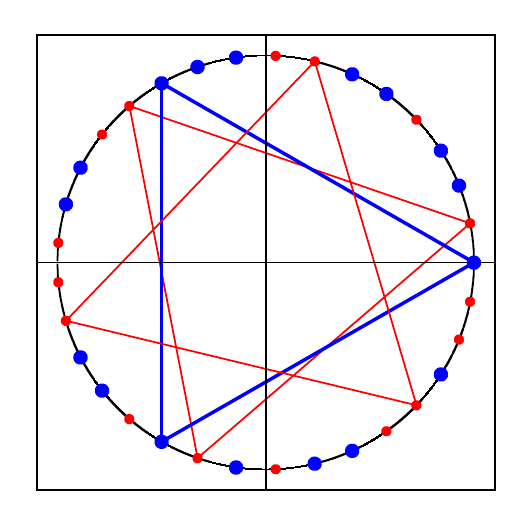}
\caption{\label{fig:2} Geometrical illustration of Lemma~\ref{lm:1} for the model vector (b) with $N=33$. Plotted are the points $e^{\cu\dft n}$ on the unit circle. The small red and large blue dots are obtained for the values of $n$ such that $x_n=1$ and $x_n=0$, respectively. There are two $3$-gons contained in the model vector (b), which are shown by thin red lines. There is also one empty $3$-gon shown by thick blue lines. This gives two ways in which a pair of polygons can be formed; each distinct pair results in a distinct false solution. The two false solutions shown in Fig.~\ref{fig:1} are obtained by switching 1s in the positions corresponding to the vertices of one of the red $3$-gons with 0s in the positions corresponding to the vertices of the blue (empty) $3$-gon.}
\end{figure}

\begin{figure*}
%FIG 3
\centering
\includegraphics[]{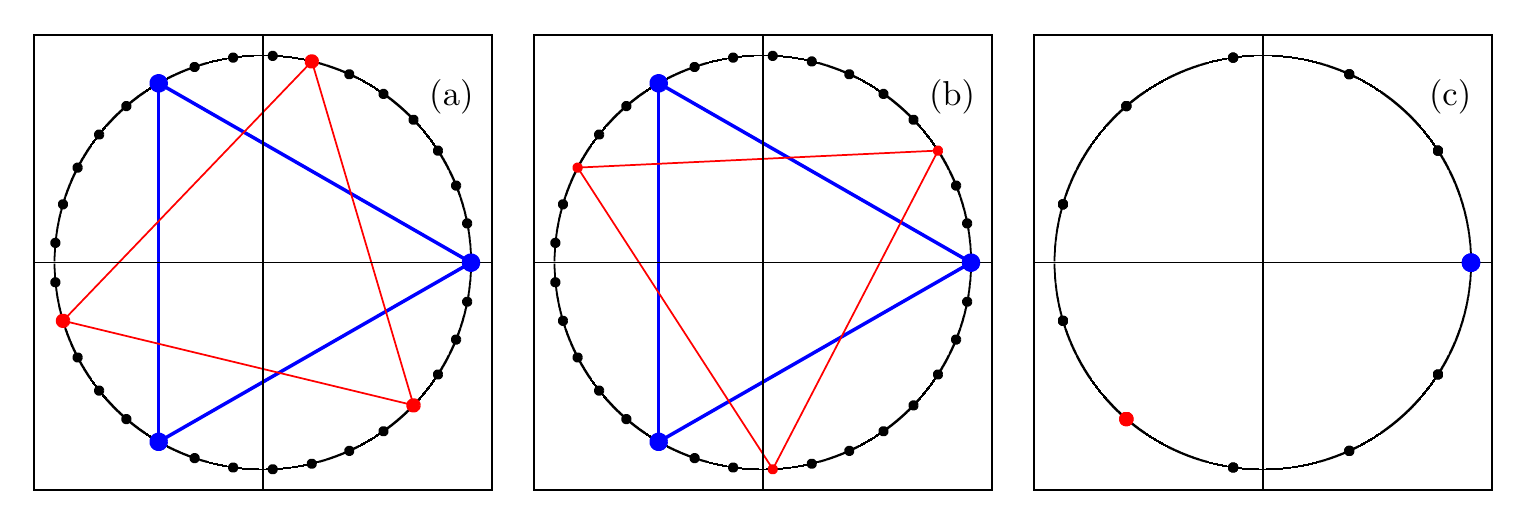}
\caption{\label{fig:3} Geometrical illustration of Lemma 2. Let ${\tt x}$ be given by the model vector (b) ($N=33,r=16$) and let ${\tt y}$ be one of the two vectors that are distinct but $1$-indistinguishable from ${\tt x}$ (we have chosen for ${\tt y}$ the specific false solution obtained by switching 1s and 0s in the pair of $3$-gons shown in Fig.~\ref{fig:2} that are geometrically farther apart).  Let, as in the proof of Theorem~\ref{th:1}, ${\tt z}={\tt x}-{\tt y}$. Plotted are the terms that enter the definition \eqref{DFT_for} of $\tilde{z}_1$ (a), $\tilde{z}_2$ (b), and $\tilde{z}_3$ (c). Thus, Panel (a) displays the terms $z_ne^{\cu\dft n}$ for $1\leq n\leq 33$. The terms with such $n$ that $z_n=0$, $z_n=1$ and $z_n=-1$ are represented by small black dots, intermediate-size red dots, and large blue dots, respectively. Panel (b) displays the terms $z_ne^{\cu \dft 2 n}$ according to the same color convention. The additional factor of $2$ in the exponent results in a permutation of the dots that are displayed in Panel (a); however, the $3$-gons are preserved. Note that the total number of dots in Panels (a) and (b) is the same and equal to $33$. It is clear that $\tilde{z}_1=\tilde{z}_2=0$ since regular $3$-gons sum to zero. In Panel (c), the terms $z_ne^{\cu \dft 3 n}$ are shown. For $1\leq n \leq 33$, the above terms take only $11$ distinct values so that each displayed dot corresponds to a sum of three terms with different $n$. However, only dots of the same color (same value of $z_n$) can overlap in this example. It can be seen that each $3$-gon of Panel(a) collapses to a single point in Panel (c). Therefore, $\tilde{z}_3 \neq 0$ implying that ${\tt x}$ and ${\tt y}$ are $3$-distinguishable.}
\end{figure*}

While Lemma~\ref{lm:1} establishes the necessary and sufficient condition for two distinct vectors in $\Omega(N,r)$ to be $1$-indistinguishable (in the two prime factors case), the following Lemma provides a potential remedy to the non-uniqueness. Specifically, it tells us how many additional DFT coefficients must be included in the data set to guarantee uniqueness.

\begin{lemma}
\label{lm:2}
Let $N=pq$ with the same conditions on $p,q$ as in Lemma~\ref{lm:1}, and let ${\tt x} \in \Omega(N,r)$. Let $L$ be the smallest integer for which ${\tt x}$ contains a pair of $L$-gons, which are not subsets of any larger-order pair of regular polygons. Then ${\tt x}$ is $L$-distinguishable from all other vectors in $\Omega(N,r)$. Moreover, $L$ is the smallest value of $k$ for which ${\tt x}$ is $k$-distinguishable from all other vectors in $\Omega(N,r)$.
\end{lemma}

The proof of Lemma~\ref{lm:2} is given in Appendix~\ref{app:B}. The geometric concept behind this proof is illustrated in Fig.~\ref{fig:3}. Lemma~\ref{lm:2} implies that, if $p\neq q$, then all vectors in $\Omega(N,r)$ are $q$-distinguishable. For $N = p^2$, all vectors are $p$-distinguishable. Considering again the case $N=33$, $r=16$, the best one can assume is that all vectors in $\Omega(33,16)$ are $11$-distinguishable. This implies that one must use $L \ge 11$ to guarantee uniqueness of IP($33,r,L$) for any $r \ge 11$. However, the model vector (b) does not contain any $11$-gons. Therefore, it can be recovered with $L=3$; using $L=11$ would be an overkill. We note that vectors that contain regular $11$-gons make a small subset of $\Omega(33,16)$ (are statistically rare). Therefore, using $L=3$ for $N=33,r=16$ entails a relatively small risk of running into a false solution. Moreover, Lemma~\ref{lm:2} tells us exactly what form these vectors, and the corresponding false solutions, take. The question of {\em statistically reliable} invertibility -- that is, accepting a small risk of finding a false solution -- is addressed more systematically in Section~\ref{sec:sta}.

In the absence of any {\it a priori} knowledge about the true solution apart from what is given in the data, we can state the following sufficient condition for uniqueness of IP($N,r,L$).

\begin{theorem}
\label{th:2}
Let $N=pq$ with the same conditions on $p,q$ as in Lemma~\ref{lm:1}, $r\leq M$, and 
\begin{align*}
L_0 = \max_{\ell,m \in {\{0,1\}}} \left(\{k = p^ \ell  q ^m \ : k\leq r\}\right) \ .
\end{align*}
Then IP($N,r,L)$ is uniquely solvable for any $L \ge L_0$.
\end{theorem}

\begin{proof}
Theorem~\ref{th:2} is a direct consequence of Lemma~\ref{lm:2}. A vector ${\tt x} \in \Omega(N,r)$ must have at least $L$ 1s to contain an $L$-gon, which implies that $L \leq r$. Then, by definition, $L_0$ is the largest order of polygons that can be formed in ${\tt x}$. Hence, by Lemma~\ref{lm:2}, any ${\tt x} \in \Omega(N,r)$ is $L_0$-distinguishable from all other vectors in $\Omega(N,r)$. Therefore, IP$(N,r,L)$ is uniquely solvable for any $L \ge L_0$.
\end{proof}

To illustrate Theorem~\ref{th:2}, consider the case $N=143=11 \cdot 13$. If $r<11$, then we have $L_0=1$. If $11 \leq r < 13$, then $L_0=11$, and if $13 \leq r \leq M = 71$, then $L_0=13$. Thus, $L=13$ guarantees uniqueness of the inverse solution for any binary vector of length $N=143$.

\begin{remark}
\label{rm:1}
Even if IP($N,r,L$) is not uniquely solvable, some vectors in $\Omega(N,r)$ can be uniquely recovered from the knowledge of $\tilde{x}_m$ with $m=1,2,\ldots L$. For example the following vector
\begin{align*}
&(c) \ N=35, \ r=17, \ S[\Omega(N,r)] = 4,537,567,650 \\
&{\tt x} = (1 0 0 1 0 1 1 0 0 0 0 1 1 1 1 0 1 1 0 0 0 1 1 0 1 0 1 0 0 1 0 0 0 1 1)
\end{align*}
is uniquely recoverable with $L=1$ even though IP($35,17,1$) is not uniquely solvable. The reason is that ${\tt x}$ does not contain any pairs of 5- or 7-gons.
\end{remark}

\begin{remark}
\label{rm:2}
Similar but more complicated results hold for the case when $N$ has two prime divisors, i.e., $N=p^\alpha q^\beta$ and the integers $\alpha,\beta$ can be larger than $1$. Here the difficulty grows from the possibility of an intricate overlapping of different polygon pairs. The case when $N$ has three or more prime divisors is even more difficult to analyze due to the existence of the so-called asymmetrical minimal vanishing sums of $N$-th roots of unity~\cite{conway_1976_1,lam_2000_1}.
\end{remark}

\section{Stability of inversion}
\label{sec:sta}

Even when an inverse problem IP($N,r,L)$ is uniquely solvable, it is not clear whether finding the solution is a numerically stable procedure. Consider the data points in Fig.~\ref{fig:1} for $N=31$, $r=15$ and $L=1$. The large (blue) dot corresponds to the true solution ${\tt x}_{\tt mod}$ and it is the only vector in $\Omega(31,15)$ with $\chi_2({\tt x}, {\tt x}_{\tt mod};1) = 0$, so that the inverse solutions is unique. However, the Fourier-space distance between the first runner-up to the true solution (let us call it ${\tt y}$) and the model is quite small: $\chi_2({\tt y}, {\tt x}_{\tt mod}; 1) \approx 0.0002$. On the other hand, the real-space distance between ${\tt y}$ and ${\tt x}_{\tt mod}$ is not small: $d({\tt y}, {\tt x}_{\rm mod}) = 10$. In other words, $10$ out of $16$ 1s in ${\tt y}$ are in the wrong places. This is an obvious sign of instability. A small change in the DFT data can result in a large change of the inverse solution. 

In this section, we investigate the stability of IP($N,r,L$) numerically. To this end, we have selected some values of $N$ and generated sets of random vectors ${\tt x}_j \in \Omega(N,r)$ for $r=2,3,\ldots M$ and $j=1,2,\ldots J$, where $J$ was chosen to be sufficiently large to obtain statistically significant results. The vectors ${\tt x}_j$ were generated as follows. For a particular random realization, $r$ 1s were randomly placed into $N$ possible positions. Repetitions (identical random realizations) were allowed but occurred very rarely. The number of random realizations in a set, $J$, depended on $r$ and $N$ and varied from $10^4$ to $10^2$ in the most difficult cases such as $N=35$, $r=17$. 

Then, for each ${\tt x}_j$, we have computed the Fourier-space distances~\footnote{In this section, we rely on the $L_2$ norm.} $\chi_2({\tt x}_j, {\tt y}; L)$ to {\em all} vectors ${\tt y} \in \Omega(N,r)$ and the minimum distance between ${\tt x}_j$ and any ${\tt y}$ that is not equal to ${\tt x}_j$. The latter quantity can be formally defined as
\begin{align}
\label{kappa_def}
\kappa({\tt x}_j;L) = \min_{{\tt y}, {\tt y}\neq {\tt x}_j} \chi_2({\tt x}_j, {\tt y}; L) \ , \ \ {\tt x}_j,{\tt y} \in \Omega(N,r) \ .
\end{align}
Note that $\kappa({\tt x}_j; L)$ and its averages defined below in \eqref{kappa_sigma_av} depend implicitly on $N$ and $r$. If $\kappa({\tt x}_j; L) = 0$, the vector ${\tt x}_j$ is not uniquely recoverable with the particular value of $L$. If $\kappa({\tt x}_j; L) > 0$ but is in some sense small, then ${\tt x}_j$ is uniquely recoverable with the given $L$ but the inverse solution is numerically unstable. Generally, as $\kappa({\tt x}_j; L)$ is increased, it becomes easier to recover ${\tt x}_j$, and the precision requirements on the DFT data become less stringent.

\begin{figure*}
%FIG 4
\centering
\includegraphics[]{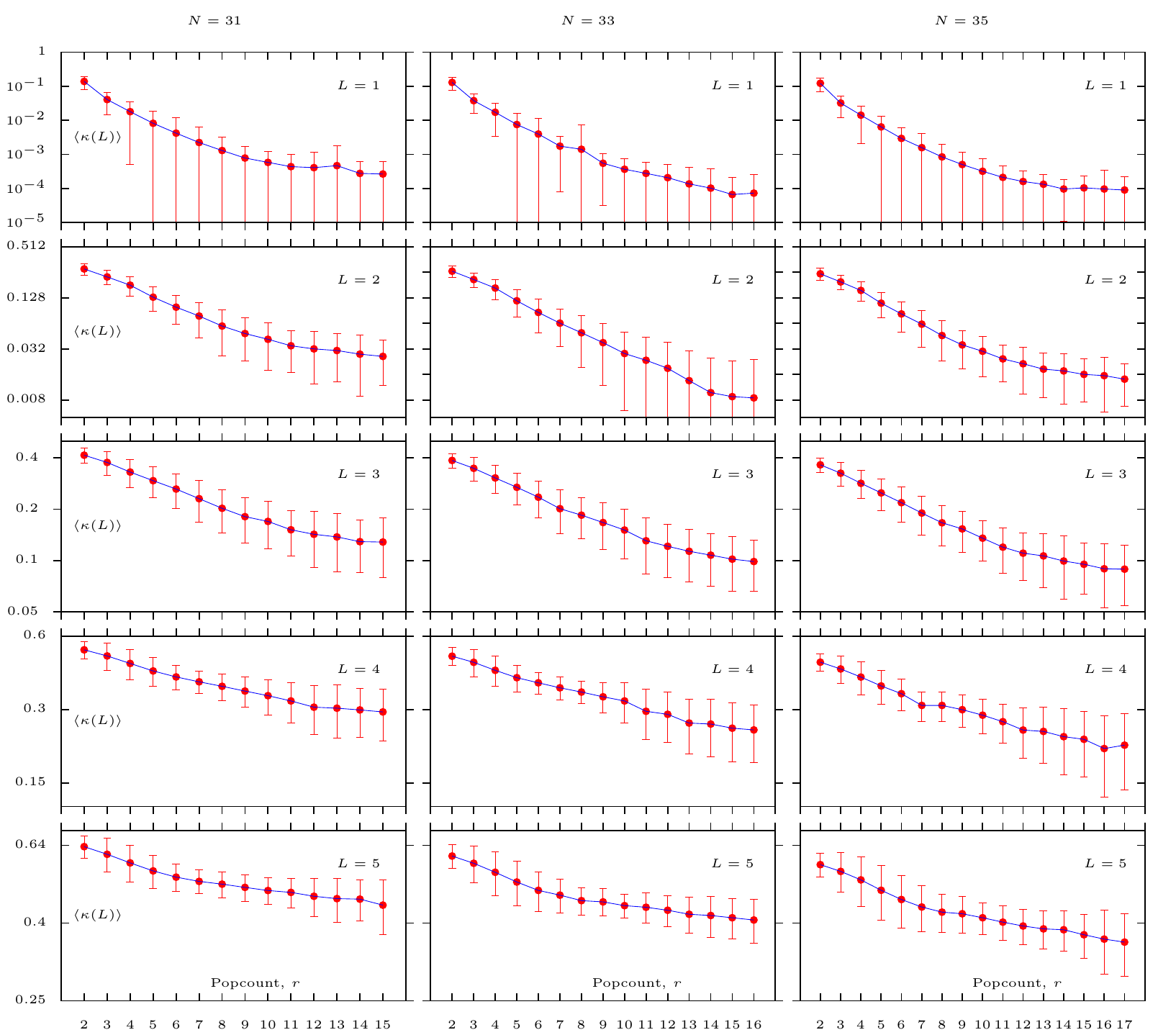}
\caption{\label{fig:4} Averages $\langle \kappa(L) \rangle$ as functions of $r$ for $N=31,33,35$ and $L$ from $1$ to $5$. Error bars are shown at the level of one standard deviation, $\sigma(L)$, as defined in~\eqref{kappa_sigma_av}. The vertical axes in all plots are logarithmic. The trivial case $r=1$ is not shown.}
\end{figure*}

\begin{figure*}
%FIG 5
\centering
\includegraphics[]{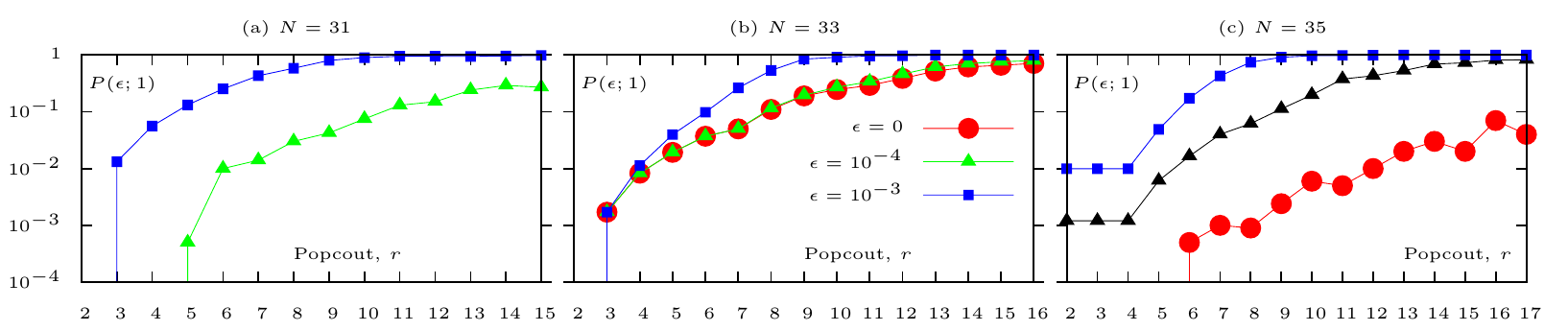}
\caption{\label{fig:5} Cumulative probabilities $P(\epsilon;1)$ for $\epsilon$ as labeled and $N=31$ (a), $N=33$ (b), and $N=35$ (c) as functions of $r$. The trivial case $r=1$ is not shown.}
\end{figure*}

We have also computed the average and the standard deviation of $\kappa({\tt x}_j;L)$ according to
\begin{subequations}
\label{kappa_sigma_av}
\begin{align}
\label{kappa_av}
& \langle \kappa(L) \rangle = \frac{1}{J} \sum_{j=1}^J \kappa({\tt x}_j; L) \ , \\ 
\label{kappa_2_av}
& \langle \kappa^2(L) \rangle = \frac{1}{J} \sum_{j=1}^J \kappa^2({\tt x}_j; L) \ , \\ 
\label{sigma}
& \sigma^2(L) = \langle \kappa^2(L) \rangle - \langle \kappa(L) \rangle^2 \ .
\end{align}
\end{subequations}
These quantities are illustrated in Fig.~\ref{fig:4} for $N=31,\ 33,\ 35$ and $L$ from $1$ to $5$. The data displayed in this figure convey how ``easy'' it is to reconstruct a vector from $\Omega(N,r)$. For example, consider the case $N=35$, $L=3$ and $r=15$. We have for these parameters $\langle \kappa \rangle\approx 0.1$ and $\sigma\approx 0.03$. This means that most vectors in $\Omega(35,15)$ (those within the the $\pm 2\sigma$-interval in the statistical distribution) have the Fourier-space distance to the closest distinct neighbor between $0.04$ and $0.16$. Therefore, if we find a vector ${\tt y} = \Omega(35,15)$ with $\chi_2({\tt x}, {\tt y}; 3) \leq 0.04$, it is likely to be the true solution. We thus conclude that the Fourier data should be specified with the absolute precision of $0.04$ or better.

Not all combinations of $N$, $L$ and $r$ allow such simple considerations. The differences $\langle \kappa\rangle - \sigma$ can be very small or even negative (this is technically possible). In such cases, the lower parts of the error bars in Fig.~\ref{fig:4} are outside of the plot frames. Relevant examples include $N=35$, $L=1$ for $r \ge  5$. In practical terms, the inverse problems with these combinations of $N$, $L$ and $r$ are hard to solve since it is likely that a vector in $\Omega(N,r)$ has a distinct neighbor that is {\em almost} $L$-indistinguishable from itself. To make the inverse problem better conditioned, one can either increase $L$ or increase the precision of the DFT data, assuming the inverse solution is unique.

So far, we have not characterized the statistical distribution of $\kappa({\tt x}_j;L)$. Therefore, the considerations based on the data of Fig.~\ref{fig:4} are qualitative. Computing the full statistical distribution of $\kappa({\tt x}_j;L)$ is a combinatorially hard problem. Instead, we have introduced the cumulative probabilities
\begin{align}
\label{P_def}
P(\epsilon; L) = \frac{1}{J} \sum_{j=1}^J \Theta\left( \epsilon - \kappa({\tt x}_j;L) \right) \ , 
\end{align}
where
\begin{align}
\label{Theta_def}
\Theta(x) = \left\{ 
\begin{array}{ll}
1 \ , & x \ge 0 \\
0 \ , & x < 0
\end{array}\right.
\end{align}
is the step function. Just like $\kappa({\tt x}_j;L)$, $P(\epsilon;L)$ depends implicitly on $N$ and $r$. It gives the (approximate) fraction of vectors in $\Omega(N,r)$ with at least one distinct neighbor that is at least as close in Fourier space (as quantified by the $\chi_2({\tt x},{\tt y};L)$) as $\epsilon$. In particular, $P(0;L)$ gives the fraction of vectors in $\Omega(N,r)$ that  have at least one distinct but $L$-indistinguishable neighbor.

The quantities $P(\epsilon;1)$ are shown in Fig.~\ref{fig:5} for $N=31,\ 33,\ 35$, and $\epsilon=0,\ 10^{-4},\ 10^{-3}$. As expected, $P(0;1)=0$ for $N=31$. This means that, in agreement with Theorem~\ref{th:1}, all vectors in $\Omega(31,r)$ with $r=1,2,\ldots,M$ are $1$-distinguishable from each other. 

Data for larger $L$ are not shown in Fig.~\ref{fig:5} but can be described as follows. In the case $N=31$, we have $P(10^{-3};L) = P(10^{-4};L) = P(0;L) = 0$ for all $r$ considered and $L>1$ (note that  $P(\epsilon;L)$ is a non-decreasing function of $\epsilon$). Therefore, choosing $L=2$ already makes the inverse problem stable in this case. 

In the case $N=33$, $P(10^{-3};L) = P(10^{-4};L) = P(0;L) = 0$ (for all $r$) when $L>2$, whereas choosing $L=2$ is still not sufficient for making the inverse problem stable. Note that, for $L<11$ and $r>11$, $P(\epsilon;L)$ are not exactly zero due to the possibility of forming regular $11$-gons (see Lemma~\ref{lm:1}). However, the values of $P$ are too small to be determined by the statistical approach used here. These cumulative probabilities can also be computed theoretically by counting all the elements ${\tt x}$ that allow regular $11$-gons.

Finally, in the case $N=35$, we can use Theorem~\ref{th:2} to determine whether there are false solutions. For $r<5$, there are no false solutions already with $L=1$. For $5 \leq r < 7$, false solutions are suppressed when $L \ge 5$. For $7\leq r \leq 17$, suppressing false solutions requires $L\ge 7$. However, even if $L<7$, false solutions are statistically rare. Moreover, when $L>1$, the false solutions appear to be the only source of numerical instability. Therefore, the inverse problem with $N=35$ and arbitrary $r$ can be solved in practice even with $L<7$. For example, choosing $L=3$ entails the probability of running into a false solution of the order of $0.01$ or smaller.

\section{Inversion}
\label{sec:inv}

Even if the solution to the inverse problem is unique and we know the DFT data with sufficient precision to guarantee stability, finding the solution is a nontrivial task. In particular, iterative methods that seek to optimize $\chi_p$ are unlikely to work. The reason for this is illustrated in Figs.~\ref{fig:6} and \ref{fig:7}. In Fig.~\ref{fig:6}, we plot the real-space distance $d({\tt x}, {\tt x}_{\tt mod})$ between various vectors ${\tt x} \in \Omega(N,r)$ and the model vector (a) vs the corresponding Fourier-space distance $\chi_2({\tt x}, {\tt x}_{\tt mod}; L)$. It can be seen that the two distances do not correlate well. It is possible to have small $\chi_2$ for a large $d$ and small $d$ with a large $\chi_2$. Any optimization technique that works directly with the real-space vectors ${\tt x} \in \Omega(N,r)$ and tries to reduce the error $\chi_2$ iteratively is therefore not likely to solve the problem: the iterations will inevitably end up in one of the many local minima of $\chi_2$, which, in real space, are still very far from the true solution. The difficulty is not removed if we use $\chi_\infty$ instead of $\chi_2$, as is illustrated in Fig.~\ref{fig:7}.

\begin{figure}
%FIG 6
\centering
\includegraphics[]{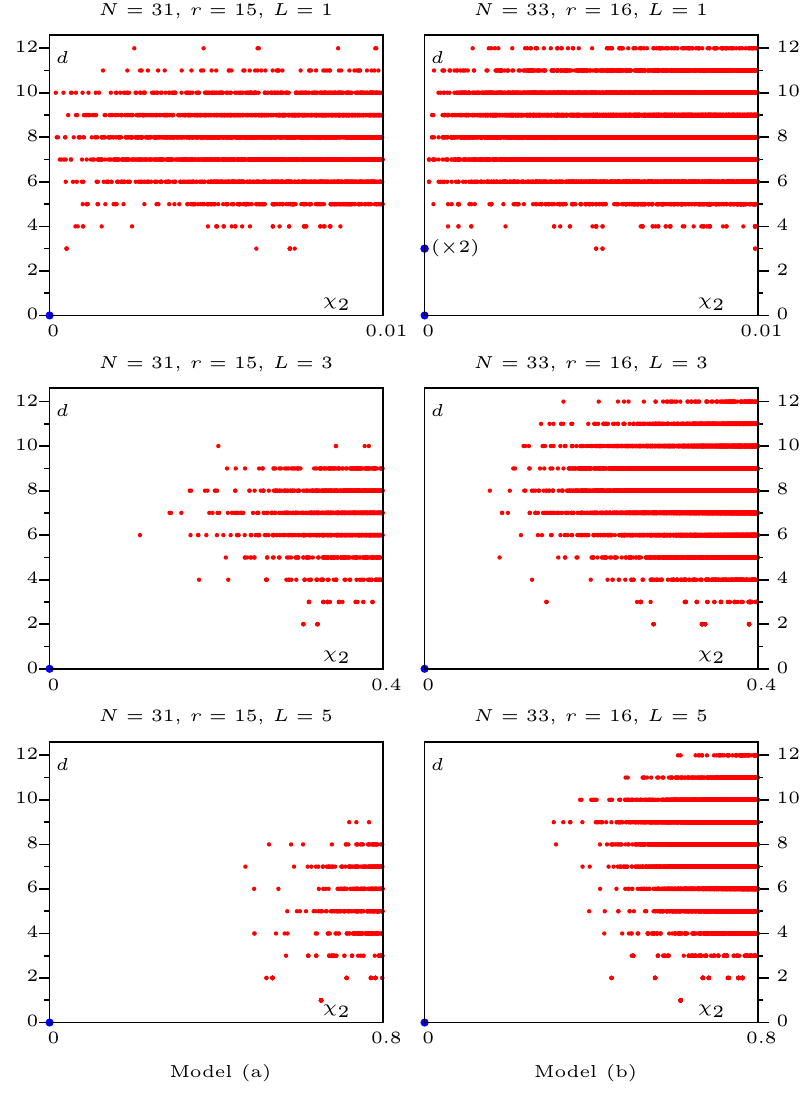}
\caption{\label{fig:6} Real-space distance $d({\tt x}, {\tt x}_{\rm mod})$ vs Fourier-space distance $\chi_2({\tt x}, {\tt x}_{\tt mod}; L)$ for the same data points as in Fig.~\ref{fig:1}. Here ${\tt x}_{\tt mod}$ is the model vector (a). The symbol $(\times 2)$ indicates that the data point corresponds to two distinct false solutions. Both false solutions have the same real-space distance to the model, $d=3$.}
\end{figure}

\begin{figure}
%FIG 7
\centering
\includegraphics[]{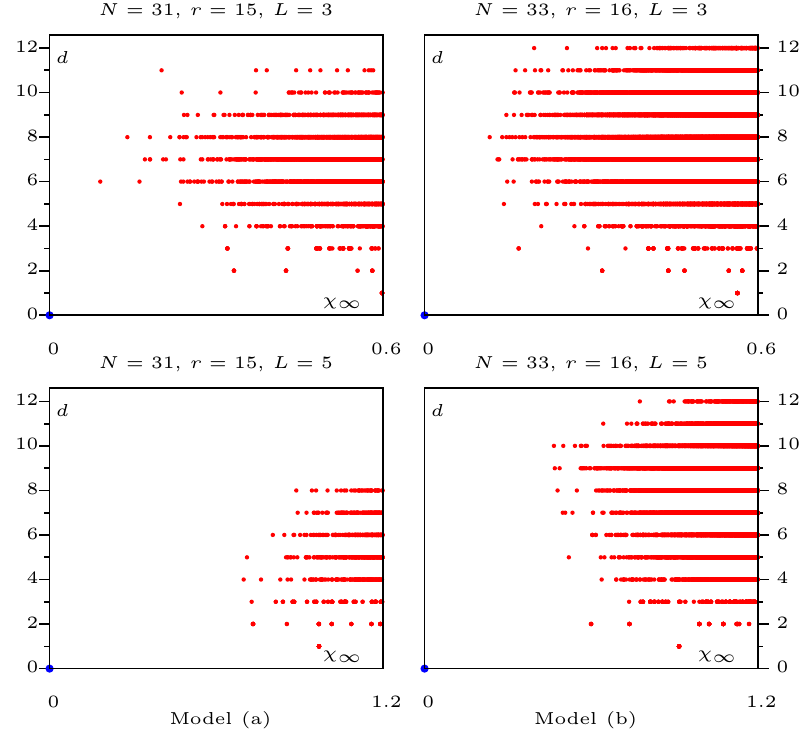}
\caption{\label{fig:7} Same as in Fig.~\ref{fig:6} but for $\chi_\infty$. Data for $L=1$ are not shown since, in this case, $\chi_2$ and $\chi_\infty$ coincide.}
\end{figure}

\begin{remark}
\label{rm:3}
If we could find a vector ${\tt y}$ such that $d({\tt y}, {\tt x}_{\rm mod}; L) = 1$, 
we would be able to find, starting from this result, the true solution (assuming it is unique) in at most $r(N-r)$ deterministic steps. But as can be seen from the data of Figs.~\ref{fig:6} and \ref{fig:7}, small $\chi_p$ does not imply small $d$. It is almost as hard to find a vector ${\tt y}$ with the above property as it is to find ${\tt x}_{\rm mod}$ itself.
\end{remark}

Thus, we have encountered a somewhat paradoxical result. The inverse problem can be linear and have a unique and stable solution; yet, any method that updates iteratively ${\tt x} \in \Omega(N,r)$ in an attempt to minimize $\chi_p$ is not expected to work. The mathematical reason for this difficulty is that the forward DFT maps a discrete set $\Omega(N,r)$ onto a continuous linear space of Fourier data; not every point in the data space is an image of an element in $\Omega(N,r)$. Iterative methods that work well for continuous maps do not work here.

However, we describe below two approaches to solving the inverse problem that work reasonably well. One is a combinatorial approach to solve the NP-hard problem, and the other is based on non-convex optimization of a continuous map; the complexity of this method is polynomial (if it converges). The combinatorial method does not seek a sequence of vectors ${\tt x}$ with monotonously decreasing value of $\chi_p$. Rather, it tests all vectors in some real-space vicinity of an initial guess (defined below) and either finds one with $\chi_p$ below a pre-determined threshold or finds the vector with the smallest $\chi_p$ over all vectors in this vicinity. The optimization method does not work with $\chi_p$ directly but rather relaxes the assumption of binarity first and defines a functional that quantifies how close a general vector is to being binary. The functional is non-convex but a fourth-order polynomial; along any search direction it can have no more than two minima. This, together with stochastic jumps strategy, allows one to search efficiently for the deepest local minimum. The Supplemental Material contains a computational package that implements these two methods. The package is applicable to generic data consisting of several DFT coefficients of the unknown vector (to be supplied by the user), and includes detailed documentation and examples.

We start by describing the set up of the numerical inverse problem in more detail. Let $\tilde{\varphi}_m$, $1 \leq m \leq L$ be a set of DFT coefficients, which are given as input to the numerical inversion. This set can be expanded by using the relations $\tilde{\varphi}_0 = r$ and $\tilde{\varphi}_{-m} = \tilde{\varphi}_m^*$. For the algorithms described below, it is not important how $\tilde{\varphi}_m$ were generated. These can be exact DFT coefficients of some binary vector, or noisy approximations to such coefficients, or just an arbitrary set of complex numbers. In the numerical package that accompanies this paper, $\tilde{\varphi}_m$ with $1\leq m \leq L$ are supplied as numbers in an input file, and several examples of such ``forward data'' corresponding to some model binary vectors (with various degrees of precision) are provided for testing. We seek a binary vector ${\tt s}$ (the solution) such that the Fourier-space distance 
\begin{align}
\label{stop_cond_1}
\chi_p({\tt s}, \upvarphi; L) \leq \epsilon \ ,
\end{align}
where $\epsilon$ is a pre-determined small constant, or, if \eqref{stop_cond_1} can not be met, we seek the minimizer of $\chi_p({\tt x}, \upvarphi; L)$ over a sufficiently large set of ${\tt x}$. These conditions are referred to below as the first and second stop conditions. We emphasize that ${\tt s}$ is a numerical solution; it may or may not be equal to the model vector ${\tt x}_{\rm mod}$ that was used to generate the data $\tilde{\varphi}_m$. It might also be the case that a model vector ${\tt x}_{\rm mod}$ was not even used to generate $\tilde{\varphi}_m$.

We also note the following three points. First, computation of $\chi_p({\tt x}, \upvarphi; L)$ according to \eqref{chi_def} does not require the knowledge of the full vector $\upvarphi$; the set of known DFT coefficients $\tilde{\varphi}_m$ with $|m| \leq L$ is sufficient. Second, if $\tilde{\varphi}_m$ do not correspond to some binary vector with given $N$ and $r$ precisely, then $\chi_p({\tt s}, \upvarphi; L)$ can not be arbitrarily small and the first stop condition can not be met if the selected $\epsilon$ is too small. However, the second stop condition can still provide the correct solution. Under the circumstances, increasing $\epsilon$ can make the numerical inversion more efficient since the first stop condition, if achievable, is generally met much faster than the second condition. Finally, the achieved distance $\chi_p({\tt s}, \upvarphi; L)$ is an indicator of how reliable the obtained solution is. This quantity can be compared to the data similar to those shown in Fig.~\ref{fig:4} but computed for the specific values of $N$, $r$ and $L$ that were used in a reconstruction. If $\chi_p({\tt s}, \upvarphi; L) < \langle \kappa(L) \rangle - 2 \sigma(L)$, one can be reasonably confident that ${\tt s}$ is the true solution.

In both approaches described below, we start with an initial guess ${\tt g} = (g_1,\ldots,g_N)$, which is computed as the low-path filtered inverse DFT of the forward data, viz,
\begin{align}
\label{guess}
g_n = \frac{1}{N}\sum_{m=-L}^L \tilde{\varphi}_m e^{-\cu \dft n m} \ , 
\end{align}
where we have used all the available data points $\tilde{\varphi}_m$ including $\tilde{\varphi}_0 = r$ and $\tilde{\varphi}_{-m}=\tilde{\varphi}_m^*$.

\subsection{Combinatorial algorithm}
\label{sec:inv.comb}

For relatively small values of $N$ (i.e., $\lesssim 60$), we can use the following approach. We start with the 
the initial guess \eqref{guess} and round off the $r$ largest elements of ${\tt g}$ to $1$ and the rest to $0$. We refer to this procedure as to ``roughening'' and write ${\tt b} = {\mathcal R}[{\tt g}]$, where ${\mathcal R}[\cdot]$ is the roughening operator. The result is a binary initial guess ${\tt b}$ with the correct length and popcount, so that ${\tt b} \in \Omega(N,r)$. However, ${\tt b}$ is not expected to be consistent with the data. To be sure, we always check whether ${\tt b}$ satisfies the first stop condition \eqref{stop_cond_1}. If this is not so, as is usually the case, we invoke a recursive procedure that builds all vectors ${\tt x}$ such that $d({\tt x}, {\tt b}) = 1$, then all vectors such that $d({\tt x}, {\tt b}) = 2$, etc. The algorithm stops when either a vector ${\tt x}$ satisfying \eqref{stop_cond_1} is found (first stop condition) or all vectors ${\tt x}$ such that $d({\tt x}, {\tt b}) \leq d_{\tt max}$, where $d_{\tt max}$ is the maximum depth of recursion, have been tested (second stop condition). If we select $d_{\tt max} = r$, then all vectors in $\Omega(N,r)$ are tested, but such exhaustive search is rarely necessary. In the computational package accompanying this paper, the default value is $d_{\tt max} = \min(10,r)$, and it can be tuned by the user if necessary. 

We now adduce the pseudo-code for the combinatorial algorithm.

\begin{algorithmic}[1]

\STATE Compute ${\tt g}$ according to \eqref{guess} and ${\tt b}$ as ${\tt b} = {\mathcal R}[{\tt g}]$;

\FOR{$d=1$ \TO $d=d_{\tt max}$}

\STATE Initialize $\chi_{\rm min} \leftarrow 10^9$;

\FOR{$i=1$ \TO $r!/d!(r-d)!$}

\STATE{\label{Step:A} Select a new unique combination of $d$ 1s out of $r$ 1s in ${\tt b}$;}

\FOR{$j=1$ \TO $(N-r)!/d!(N-r-d)!$}

\STATE{\label{Step:B} Select a new unique combination of $d$ 0s out of $N-r$ 0s in ${\tt b}$;}

\STATE{Swap 1s selected in Line~\ref{Step:A} with 0s selected in Line~\ref{Step:B} and leave other 1s and 0s in ${\tt b}$ unchanged; assign the resulting values to ${\tt x}$;}

\STATE{Compute $\chi \equiv \chi_p({\tt x}, \upvarphi; L)$;}

\IF{$\chi \leq \epsilon$}

\STATE{Solution found using Stop Condition 1. Assign ${\tt s} \leftarrow {\tt x}$ and exit;}

\ELSE 

\IF{$\chi < \chi_{\rm min}$}

\STATE{$\chi_{\rm min} \leftarrow \chi$;}

\STATE{${\tt x}_{\rm min} \leftarrow {\tt x}$;}

\ENDIF
\ENDIF

\ENDFOR
\ENDFOR
\ENDFOR

\STATE{Solution found using Stop Condition 2. Assign ${\tt s} \leftarrow {\tt x}_{\rm min}$;}

\PRINT{Achieved distance to data, $\chi_{\rm min}$;}

\PRINT{Recursion depth at which solution was found, $d_{\rm min}$;}

\end{algorithmic}

The two internal loops, which go over all unique $d$-combinations of $r$ 1s and $N-r$ 0s in ${\tt b}$ can be defined recursively. The problem here is that of generation of all $d$-combinations of a set of $r$ or $N-r$ distinguishable (labeled) objects. To construct all $d$-combinations of 1s, we define the recursive procedure ${\rm next\_1}(k,l)$, where $k$ is the number of 1s already selected and $l$ is the sequential number of the previous 1 selected. Here we assume that all 1s in ${\tt b}$ are numbered sequentially from $1$ to $r$ (one can imagine a label attached to each 1). Similarly, 0s can be numbered sequentially from $1$ to $N-r$. The procedure is first invoked as ${\rm next\_1}(0,0)$. For generic $k$ and $l$, ${\rm next\_1}(k,l)$ goes over all available positions $i$ for the next, $(k+1)$-th, 1. These positions, $i$, are in the range $l < i \leq r-d + (k+1)$. For each $i$, the procedure includes the corresponding 1 into a new combination and then invokes itself as ${\rm next\_1}(k+1,i)$. Importantly, the new 1 in a combination is always selected ``to the right'' of the previous selection. Once a full $d$-combination of 1s is constructed, we have $k=d$, and ${\rm next\_1}(d,*)$ invokes another recursive procedure ${\rm next\_0}(0,0)$, which builds all unique $d$-combination of 0s in a similar manner. Once two unique $d$-combinations have been constructed, the 0s and 1s in ${\tt b}$ are swapped and $\chi_p$ is computed. 

The computational complexity of the described algorithm scales as $O(L C)$, where
\begin{align}
\label{C_def}
C = \sum_{d=1}^{d_{\tt max}} \frac{r!(N-r)!}{(d!)^2(r-d)!(N-r-d)!}  \ .
\end{align}
Every term in this summation is the product of the number of unique $d$-combinations of $r$ 1s times the number of unique $d$-combinations of $N-r$ 0s, which are indicated in Lines 4 and 6 of the pseudo-code. This product is equal to the number of all unique vectors ${\tt x}$ tested at the recursion depth $d$. It is also equal to the number of all vectors in $\Omega(N,r)$ whose real-space distance to ${\tt b}$ is $d$. The factor of $L$ in $O(LC)$ accounts for the overhead of computing $\chi_p$ for every ${\tt x}$ tested. The complexity $C$ is illustrated in Fig.~\ref{fig:8} for several values of $N$ and $r$ as a function of $d_{\tt max}$. It can be seen that $C$ is much smaller than the complexity of exhaustive search [that is, testing all vectors in $\Omega(N,r)$] if $d_{\tt max}$ is significantly smaller than $r$. For example, for $N=61$, the complexity is still manageable and well below that of exhaustive search for $d_{\tt max} \lesssim 10$. 

Still, the complexity \eqref{C_def} strongly depends on $d_{\tt max}$, and the choice of this parameter in practical computations is not trivial. Let us assume that the data $\tilde{\varphi}_m$ correspond to some binary vector ${\tt x}_{\tt mod}$ with good precision. That is, the stop condition $\chi_p({\tt x}, \upvarphi; L) \leq \epsilon$ is met for ${\tt x} = {\tt s} = {\tt x}_{\tt mod}$ and is not met for any other vector in $\Omega(N,r)$. Then, to guarantee finding the true solution, the recursion must run to $d_{\tt max} = d({\tt b}, {\tt x}_{\tt mod})$. The distance from the initial guess to the model is bounded from above by $r$, but otherwise is not known {\em a priori}. This underscores the importance of making the initial guess ${\tt b}$ as close to the true solution as possible. In Fig.~\ref{fig:9}, we plot the average values $\langle d({\tt b}, {\tt x}_{\tt mod}) \rangle$ and the corresponding standard deviations for $10^6$ random model vectors ${\tt x}_{\tt mod}$ of the length $N=61$ and different values of $L$ as functions of the popcount $r$. It can be seen that the typical values of $d({\tt b}, {\tt x}_{\tt mod})$ are significantly smaller than $r$ and, as expected, decrease with $L$. We emphasize that the results shown in Fig.~\ref{fig:9} pertain to random model vectors without any particular structure. The vectors in which 1s are grouped, i.e., representing a single rectangular pulse or a few such pulses tend to have smaller $d({\tt b}, {\tt x}_{\tt mod})$. Therefore, the data of Fig.~\ref{fig:9} or similar easily-computable data sets can be useful for estimating the values of $d_{\tt max}$ that are sufficient for obtaining the solution.

\begin{figure}
%FIG 8
\centering
\includegraphics[]{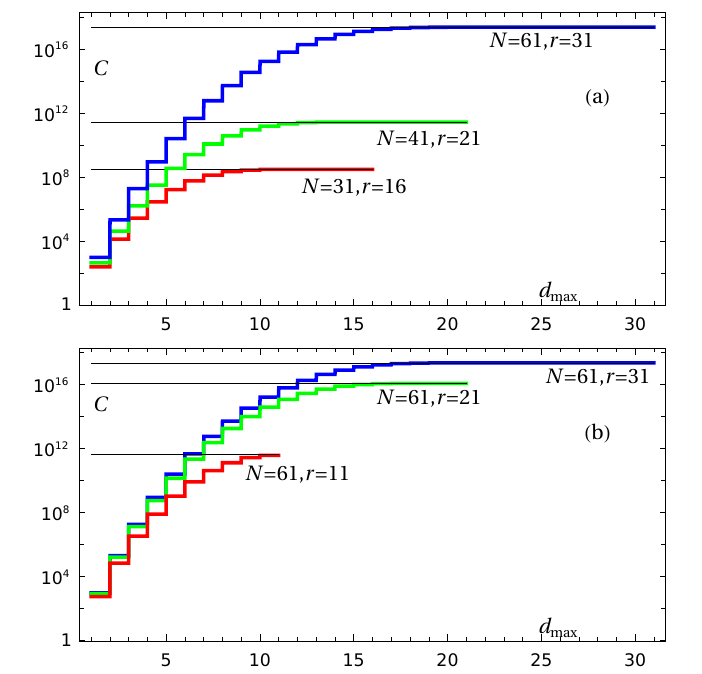}
\caption{\label{fig:8} Computational complexity $C$ \eqref{C_def} as a function of the maximum recursion depth $d_{\tt max}$ for several values of $N$ and $r$, as labeled. Thin black lines indicate the computational complexity of the exhaustive search, i.e., testing all vectors in $\Omega(N,r)$. Panel (a) shows the dependence for several values of $N$ and the maximum value of $r$ that is allowed for each $N$. Panel (b) shows the dependence for $N=61$ and several allowed values of $r$ for this $N$. The blue lines for $N=61,r=31$ in the two panels are identical.}
\end{figure}

\begin{figure*}
%FIG 9
\centering
\includegraphics[]{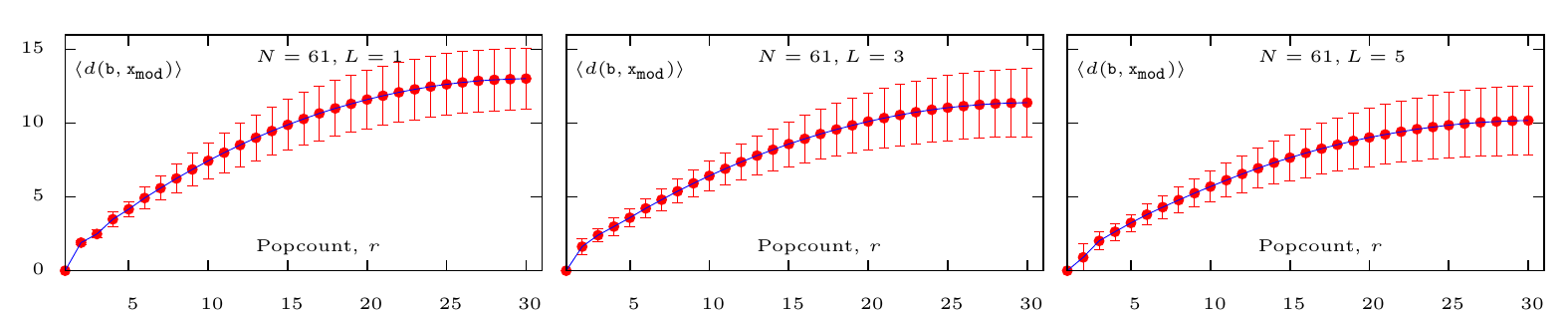}
\caption{\label{fig:9} Average distance between random models ${\tt x}_{\tt mod}$ and the corresponding (roughened) initial guess ${\tt b}$, $\langle d({\tt b}, {\tt x}_{\tt mod}) \rangle$, for $10^4$ random vectors ${\tt x}_{\tt mod}$ of length $N=61$ each, shown as functions of $r$ for different values of $L$. Error bars are drawn at the level of one standard deviation. } 
\end{figure*}
	
The algorithm described in this subsection was able to find all three model vectors (a), (b) and (c) defined above with $L=1$ by using the stop condition $\chi_2 \leq \epsilon = 10^{-5}$ (for the models (a) and (b) $\epsilon=10^{-4}$ is sufficient). We note that IP($33,16,1$) is not uniquely solvable but the algorithm has found the true solution anyway. This occurred by chance; the search could have ended up with one of the two false solutions. Also, IP($35,17,1$) is not uniquely solvable in general but, as was mentioned in Remark~\ref{rm:1}, the particular model vector (c) is uniquely recoverable with $L=1$. In Fig.~\ref{fig:10}, we show the intermediate reconstruction steps for model (c) using $L=1$ and $L=5$. It can be seen that the band-limited initial guess ${\tt g}$ does not have a well defined structure at both $L=1$ and $L=5$. The ``roughened'' binary initial guess ${\tt b}$ has the distance $d({\tt b}, {\tt x}_{\rm mod})=8$ for $L=1$ and $d({\tt b}, {\tt x}_{\rm mod}) = 4$ for $L=5$. 

For $L=1$, the combinatorial inversion took 8, 40, and 209 seconds to recover models (a), (b), and (c), respectively. In contrast, MATLAB's {\tt intlinprog} function recovered these models in 60, 518, and 3006 seconds. See the User Guide of the computational package (in Supplemental Material) for additional benchmarks and examples.

\begin{figure}
%FIG 10
\centering
\includegraphics[]{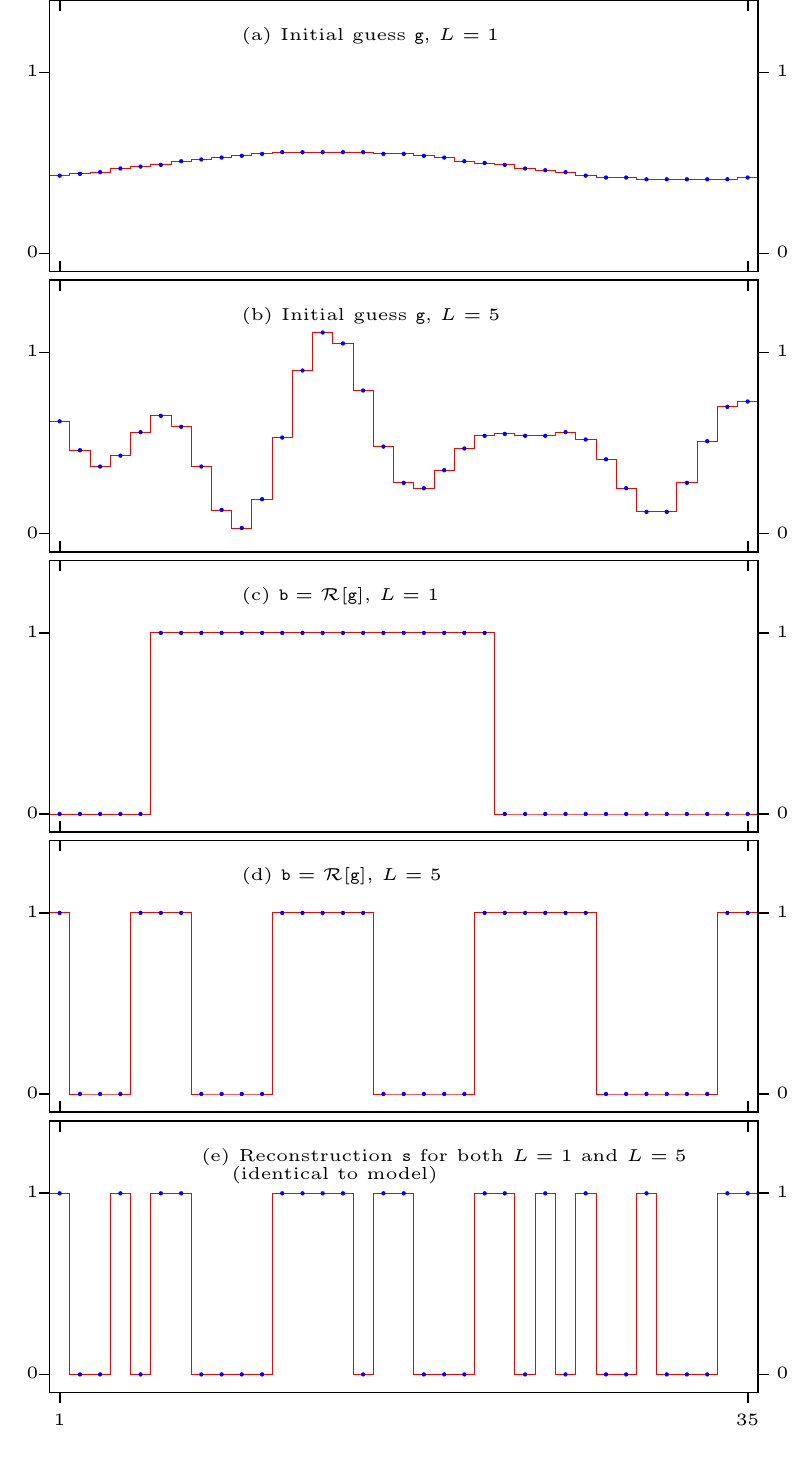}
\caption{\label{fig:10} Intermediate steps in the reconstruction of model (c) with $N=35$, $r=17$ and $L=5$. Initial guesses for $L=1$ ${\tt g}$ (a) and $L=5$ (b); initial guess after roughening, ${\tt b}=R[{\tt g}]$, for $L=1$ (c) and $L=5$ (d); and the reconstruction ${\tt s}$ (e), the same for both values of $L$ and identical to the model.}
\end{figure}

\subsection{Non-convex optimization}
\label{sec:inv.opt}

When $N$ increases past $\sim 60$, the combinatorial algorithm of the previous subsection becomes impractical. We now describe an alternative approach that relies on the continuity of the conventional (unrestricted and unconstrained) DFT to define an iterative scheme that does not involve combinatorial complexity. Let ${\tt v} = (v_1, \ldots, v_N)$ be a general real-valued unconstrained vector of length $N$. We define the cost function $F[{\tt v}]$ that quantifies how close ${\tt v}$ is to a binary vector as
\begin{align}
\label{F_def}
F[{\tt v}] = \sum_{n=1}^N v_n^2 (v_n - 1)^2 \ .
\end{align}
We then seek to minimize $F[{\tt v}]$ while keeping ${\tt v}$ consistent with the data. To this end, we start with the band-limited initial guess ${\tt v} = {\tt g}$ [defined in \eqref{guess}] and update ${\tt v}$ iteratively according to
\begin{subequations}
\label{iter}

\begin{align}
\label{iter_step}
{\tt v} \longleftarrow {\tt v} + {\tt q} \ ,
\end{align}
where ${\tt q} = (q_1, \ldots, q_N)$ is of the form
\begin{align}
\label{q_def}
q_n = \sum_{m=L+1}^M \left[c_m \cos (\dft mn) + s_m \sin (\dft m n) \right] \ .
\end{align}
\end{subequations}
Here $c_m$ and $s_m$ are some real-valued coefficients, which must be determined at each iteration step independently. It can be seen that the vector ${\tt v}$ updated according to \eqref{iter} remains consistent with the data. To determine ${\tt q}$, we use the steepest descent approach. The derivatives of $F[{\tt v} + {\tt q}]$ with respect to the set of coefficients $c_m, s_m$ evaluated at ${\tt q}=0$ are given by
\begin{subequations}
\label{dir_cm}
\begin{align}
F_m^{(c)}[{\tt v}] \equiv \left. \frac{F[{\tt v} + {\tt q}]}{\partial c_m} \right|_{{\tt q}=0} = 2\sum_{n=1}^N u_n \cos(\dft m n) \ , \\
F_m^{(s)}[{\tt v}] \equiv \left. \frac{F[{\tt v} + {\tt q}]}{\partial s_m} \right|_{{\tt q}=0} = 2\sum_{n=1}^N u_n \sin(\dft m n) \ ,
\end{align}
\end{subequations}
where
\begin{align}
\label{u_def}
u_n = v_n(v_n - 1)( 2 v_n -1 ) \ .
\end{align}
Therefore, we have determined the gradient of $F[{\tt v}]$ with respect to the unknown coefficients $c_m$ and $s_m$. Denoting the gradient vector by ${\tt p} = (p_1,\ldots,p_N)$, we have for the components:
\begin{align}
\label{p_def}
%\hspace*{-1mm} 
p_n = \sum_{m=L+1}^M \left[ F_m^{(c)}[{\tt v}]\cos(\dft m n) + F_m^{(s)}[{\tt v}]\sin(\dft m n) \right] \ .
\end{align}
According to the general strategy of steepest descent, we select ${\tt q} = \alpha {\tt p}$ in the iteration step \eqref{iter_step}, where $\alpha$ is a scalar to be determined. By direct substitution, we find that 
\begin{subequations}
\label{F_der_alpha}
\begin{align}
\label{f_alpha}
f(\alpha) \equiv \frac{\partial F[{\tt v}+\alpha{\tt p}]}{2\partial \alpha} = A_0 + A_1 \alpha + A_2 \alpha^2 + A_3 \alpha^3 \ ,
\end{align}
where
\begin{align}
& A_0 = \sum_{n=1}^N v_n(v_n - 1)(2v_n - 1) p_n = \sum_{n=1}^N u_n p_n \ , \\
& A_1 = \sum_{n=1}^N [2v_n (v_n - 1) + (2v_n - 1)^2] p_n^2 \ , \\
& A_2 = 3\sum_{n=1}^N (2v_n - 1) p_n^3 \ , \ \ A_3 = 2\sum_{n=1}^N p_n^4 \ .
\end{align}
\end{subequations}
The function $F[{\tt v}+\alpha{\tt p}]$ is a fourth-order polynomial in $\alpha$ with a positive senior coefficient. Consequently, either it has one real minimum or two minima and one maximum. These special values of $\alpha$ can be determined by solving the cubic equation $f(\alpha)=0$. If the cubic has a single real root $\alpha_1$, we set $\alpha=\alpha_1$. If there are three real roots $\alpha_1 \leq \alpha_2 \leq \alpha_3$, then $F[{\tt v}+\alpha{\tt p}]$ has a maximum at $\alpha_2$ and two minima at $\alpha_1$ and $\alpha_3$. To avoid ``jumping'' over a maximum, we set $\alpha = \alpha_1$ if $\alpha_2>0$ and $\alpha=\alpha_3$ otherwise. This completely defines each iteration step and guarantees that $F[{\tt v}]$ is decreased in each iteration until a local minimum of $F[{\tt v}]$ is reached. The condition for the local minimum is
\begin{align}
\label{loc_min_cond}
& \sum_{n=1}^N u_n \cos(\dft mn) = \sum_{n=1}^N u_n \sin(\dft mn) = 0 \\
& \nonumber \hspace*{5cm} {\rm for}  \ L<m\leq M  \ ,
\end{align} 
where $u_n$ are defined in \eqref{u_def}. The steepest descent algorithm described here can find a local minimum satisfying the above condition efficiently and with high precision. Note that, in most iteration steps, the cubic has only one real root. The problem is however that $F[{\tt v}]$ has many local minima. It is unlikely that any given local minimum is close to the true solution.

To overcome the above difficulty, we can adopt the following stochastic approach. Once a local minimum $\bar{\tt v}_k$ is found ($k$ labels different local minima), we compute $\chi_p({\mathcal R}[\bar{\tt v}_k], \upvarphi; L)$
and check whether the first stop condition \ref{stop_cond_1} has been satisfied. If so, we have found the solution ${\tt s} = {\mathcal R}[\bar{\tt v}_k]$. If not, we start from the deepest local minimum found so far, $\bar{\tt v}_{\tt min}$, and perturb it by adding a vector ${\tt q}_{\tt rand}$ of the form \eqref{q_def} with a given length and random direction (in the space of coefficients $c_m$, $s_m$). In this way, we select a new initial guess for ${\tt v}$, from which we run the steepest descent algorithm again. Depending on the length and direction of ${\tt q}_{\tt rand}$, we will end up either in the same or a different local minimum. To avoid being trapped in the same local minimum, we gradually increase the length of random jumps.

Finally, a functionality is provided in the package to define the stop condition in terms of $F[{\tt v}]$ itself rather than in terms of $\chi_p$. Since the latter approach does not require computation of $\chi_p$, it is slightly faster. The default setting is however to use \eqref{stop_cond_1}. Finally, the algorithm stops after a certain amount of local minima has been found and reports the deepest local minimum (second stop condition).

A simplified pseudo-code for the non-convex optimization algorithm is presented below. Some parts of the algorithm are stated only briefly and some additional logic, which is needed to avoid mistakes and improve precision and speed, is not shown to avoid excessive complexity. However, conceptually important steps of the algorithm are all shown.

\begin{algorithmic}[1]

\STATE{Define constants: small precision-related constant $\delta$, random jump step length (initial value $S=1$), number of iterations before $S$ is incremented $i_{\tt inc}$ (typical value $i_{\tt inc}=1000$), and the increment of $S$, $\Delta$. These and other constants are initialized from input parameter files.}
 
\STATE{Compute ${\tt g}$ according to \eqref{guess};}

\STATE{Initialize ${\tt v} \leftarrow {\tt g}$;}

\STATE{Initialize $\chi_{\rm min} \leftarrow 10^9$;}

\FOR{$i=1$ \TO $i_{\tt max}$}

\STATE{Compute $F_0 = F[{\tt v}]$;}

\STATE{Initialize ${\rm test \leftarrow True}$;}

\WHILE{{\rm test}}

\STATE{Compute steepest descent direction ${\tt p} = {\tt p}[{\tt v}]$ according to \eqref{dir_cm}-\eqref{p_def};}

\STATE{Compute the length of the steepest descent step $\alpha$ by solving the cubic \eqref{f_alpha} and choosing the distance to the closest minimum along the search direction;}

\STATE{Make the steepest descent step ${\tt v} \leftarrow {\tt v} + \alpha {\tt p}$;}

\STATE{Compute $F=F[{\tt v}]$;}

\STATE{Evaluate $t \leftarrow F_0/F - 1$;}

\IF{$t > \delta$}

\STATE{$F_0 \leftarrow F$;}

\ELSE

\STATE{${\rm test = False}$;}

\ENDIF

\ENDWHILE

\STATE{Check the local minimum condition \eqref{loc_min_cond}. If local minimum can not be reached, report error and exit;}

\STATE{Compute ${\tt x} = {\mathcal R}[{\tt v}]$;}

\STATE{Compute $\chi_p = \chi_p({\tt x}, \upvarphi; L)$;}

\IF{$\chi_p < \epsilon$}

\STATE{Stop condition 1. Solution has been found. Assign ${\tt s} \leftarrow {\tt x}$ and exit;}

\ENDIF

\IF{$\chi_p < \chi_{\tt min}$}

\STATE{$\chi_{\tt min} \leftarrow \chi_p$;}

\STATE{${\tt v}_{\tt min} \leftarrow  {\tt v}$;}

\ENDIF

\IF{$\mod(i,i_{\tt inc})=0$}

\STATE{Increase the random jump step length as $S \leftarrow S + \Delta$;}

\ENDIF 

\STATE{Construct a vector ${\tt q}$ of the form \eqref{q_def} of unit length and random direction in the space of $c_m,s_m$, $L<m\leq M$;}

\STATE{Make random jump ${\tt v} \leftarrow {\tt v}_{\tt min} + S {\tt q}$;}

\ENDFOR
	
\STATE{Solution found using Stop Condition 2. Assign ${\tt s} \leftarrow {\mathcal R}[{\tt v}_{\rm min}]$;}

\PRINT{Achieved distance to data, $\chi_{\rm min}$;}

\end{algorithmic}

Some variations of this algorithm, as implemented in the computational package, include the following. (i) If local minimum in Line 20 is not verified with required precision, the code will attempt to approach the minimum closer by moving along several deterministic directions parallel to the axes or multi-dimensional diagonals in the space of $c_m, s_m$. Only if after several attempts the local minimum could not be verified, the code reports an error. This error occurs (due to round-off errors) extremely rarely and can be fixed by changing the precision-related constants. Note that the functional $F[{\tt v}]$ only has local minima and maxima but not saddle points. (ii) The random jumps in Line 34 can originate not only from the deepest local minim found (default) but also from the initial guess, i.e., ${\tt v} \leftarrow {\tt g} + S {\tt q}$ or from the local minimum most recently found, ${\tt v} \leftarrow {\tt v} + S {\tt q}$. In the latter case, the search for solution is a random walk over the local minima of $F[{\tt v}]$. In the first (default) case, the random walk is restricted so that the next stop always has a smaller $\chi_p$. (iii) The algorithm can use a different formulation of the stop condition 1, which is based on smallness of $F$ rather than $\chi_p$. (iv) Finally note that the current implementation of the codes relies on the $L_2$ norm. This can be changed to $L_1$ or $L_\infty$ norms as explained in the User Guide.

\begin{figure}
%FIG 11
\centering
\includegraphics[]{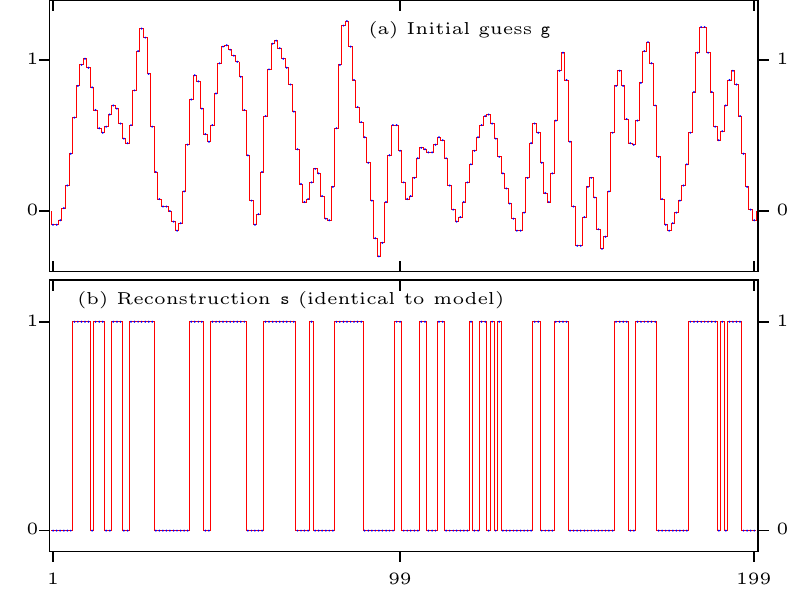}
\caption{\label{fig:10_N=199} Reconstruction with $N=199$, $r=90$ and $L=29$. Initial guess ${\tt g}$ (a) and the reconstruction ${\tt s}$ (identical to the model ${\tt x}_{\tt mod}$) (b).}
\end{figure}

To illustrate the method, we have reconstructed a vector of the length $N=199$ with $r=90$ using $L=29$ (so that the ratio $L/M=29/99 \approx 1/3$) and the stop condition $\chi_2 \leq  \epsilon=10^{-5}$ (if the DFT data are rounded-off to 4 significant figures, solution is still found fast with $\epsilon=0.002$). The initial guess ${\tt g}$ for this simulation is shown in Fig.~\ref{fig:10_N=199}(a) and the reconstruction (identical to the model) in Fig.~\ref{fig:10_N=199}(b). If we apply the roughening operation directly to the initial guess, ${\tt b} = {\mathcal R}[{\tt g}]$, we would obtain $d({\tt b}, {\tt x}_{\tt mod}) = 8$. Although the roughened initial guess and the model are not too far apart in real space, the complexity of finding the solution by the combinatorial method of Section~\ref{sec:inv.comb} is $O(10^{24})$ operations. This is beyond the reach for any modern computer. The optimization method of this subsection however finds the solution in under one minute. In general, however, it is difficult to state a general condition of convergence to the true solution or estimate the computational complexity of the method described in this subsection.  We note that the method has several adjustable parameters and changing them can help in situations when convergence appears to be slow. Additional details are provided in the User Guide of the computational package.
 
\section{Discussion and conclusions}
\label{sec:disc}

We have shown that binary compositional constraints are powerful priors that allow one to obtain a well-pronounced super-resolution effect. In particular, we have proved that a band limited DFT of a binary vector of length $N$ is uniquely invertible from the knowledge of just two complex DFT coefficients -- the zeroth and the first -- if $N$ is prime. If $N$ has two prime factors, then Theorem~\ref{th:2} tells us how many additional DFT coefficients must be known to guarantee uniqueness. The above result can be useful if applied to analysis of two-dimensional images.

Although Theorems~\ref{th:1} and \ref{th:2} establish the sufficient condition for invertibility assuming only that the unknown vector is consistent with the data, many (but not all) binary vectors are uniquely recoverable even if conditions of the Theorems do not hold.

We have further investigated stability of inversion and conditions under which reconstructing a binary vector from a limited set of DFT coefficients is practically feasible. Preliminary numerical data indicate that stable reconstructions can be obtained when about $1/3$ of all DFT coefficients are known. This entails an approximately three-fold super-resolution effect.

Although binarity is a powerful constraint, devising practical reconstruction algorithms is a nontrivial task. We provided and demonstrated two such algorithms in the paper. The first (combinatorial) algorithm is guaranteed to find a solution of the NP-hard problem if run consistently but might become prohibitive due to large computational time. Still, it involves fewer or much fewer operations than exhaustive search. This approach is applicable to vectors with $N \lesssim 100$. For longer vectors, we have developed an algorithm based on optimization of the non-convex cost function $F[\cdot]$ \eqref{F_def}. The local minima of this function can be easily found using the steepest descent approach. However, not every local minimum of a non-convex function is a global minimum. To overcome this difficulty, we have introduced random perturbations, which allow one to sample many local minima and finally find one with sufficient depth, which then coincides with the global minimum. We note that convexization of the inverse problem, that is, finding a relevant convex cost function, proved to be difficult and is likely impossible.
 
Although we have obtained numerical reconstructions for vectors up to $N=199$ with about $1/3$ of DFT coefficients considered to be known, it is clear that the developed algorithms are not optimal and leave a lot of space for improvement. We hope that, with further refinements, the theoretical results of this paper can be used to obtain the super-resolution effect in two-dimensional black-and-white images. Extending the theoretical results and algorithm performance to two-dimensional binary images is highly related to discrete tomography~\cite{herman_2012_1}. This would allow for future applications in image processing, nondestructive testing~\cite{krimmel_2005_1}, X-ray crystallography~\cite{alpers_2006_1}, and medical imaging~\cite{herman_2003_1}. One key challenge to overcome is to adapt our inversion algorithms to inversion of  two-dimensional DFT. While the $N=199$ inversion runs fast in one-dimension, further algorithm development is required for achieving fast super-resolved inversion of $199 \times 199$ images. Refinement of the cutting planes approach (applicable to both combinatorial and optimization algorithms) appears to be a promising way forward. 

\bibliography{abbrev,Master,book,local}

\appendices

\section{Proof of Lemma~\ref{lm:1}}
\label{app:A}

Suppose ${\tt x}$ is not pairwise $1$-distinguishable from some distinct vector ${\tt y}$ in $\Omega(N,r)$.  By definition, we have $\tilde{x}_1=\tilde{y}_1$.  Following the same notation as in the proof of Theorem 1, we define ${\tt z}={\tt x}-{\tt y}$ and conclude that \eqref{z1_thm1} must hold, where each $z_n$ takes a value in $\{0,\pm 1\}$. As ${\tt x}\neq{\tt y}$, not all entries of $z_n$ can be $0$. Furthermore, as ${\tt x}$ and ${\tt y}$ are binary vectors with the same popcount, by construction $\tilde{z}_0=0$. 

We will show that ${\tt z}$  either has an equivalent number of $p$-gons and ``negative" $p$-gons (a polygon with entries of $-1$), or that  ${\tt z}$  has an equivalent number of $q$-gons and ``negative" $q$-gons. As each positive $p$- or $q$-gon in ${\tt z}$ corresponds to a $p$- or $q$-gon in ${\tt x}$ (and an empty $p$- or $q$-gon in ${\tt y}$), and each negative $p$- or $q$-gon in ${\tt z}$ corresponds to an empty $p$- or $q$-gon in ${\tt x}$ (and a $p$- or $q$-gon in ${\tt y}$), this will prove the forward direction.  

We consider two cases: $p=q$ and $p\neq q$.  First, consider the case $p\neq q$.  Letting $\zeta_p$ and $\zeta_q$ be primitive roots of unity of $p$-th and $q$-th order, respectively, we rewrite \eqref{z1_thm1} using Theorem 2.3 in~\cite{lenstra_1978_1} as 
\begin{align}
\label{z1_split_lemma1}
0= \sum_{\ell=1}^q\sum_{k=1}^p  z_{(k\ell)} \zeta_p^k \zeta_q^\ell \ .
\end{align}
Here we have introduced the composite index $(k\ell)$ from 1 to $pq$.  Note that as $\zeta_p^k\zeta_q^\ell=\zeta_N^{p\ell+qk}$, by fixing either a value of $\ell$ or $k$, we vary over roots of unity the make a regular $p$- or $q$-gon, respectively.    We can now apply Lemma 1 of~\cite{conway_1976_1} to conclude that  \eqref{z1_split_lemma1} holds if and only if the inner sum is constant for all $\ell$.  That is, for all $\ell$, 
\begin{align*}
\sum_{k=1}^pz_{(k\ell)}\zeta^k_p = \sum_{k=1}^p z_{(k1)}\zeta^k_p \ .
\end{align*}
Similar to the proof of Theorem 1, as $\zeta_p$ is a root of unity of a prime order, this equation can only hold if, for each $\ell$, there is a fixed constant difference between $z_{(k\ell)}$ and $z_{(k1)}$ for all $k$. Thus we have, for all $k$ and $\ell$,
\begin{align}
\label{eq:C_ell}
z_{(k\ell)}-z_{(k1)}= C_{\ell} \ ,
\end{align}
where each $C_{\ell}\in\{0,\pm1, \pm2\}$. Now, by summing over all $k$ and $\ell$, we obtain
\begin{align*}
\sum_{k=1}^p\sum_{\ell=1}^q \left(z_{(k\ell)}-z_{(k1)}\right) = & \sum_{k=1}^p\sum_{\ell=1}^q C_\ell \ , \\ 
 -q\sum_{k=1}^p z_{(k1)} = & p\sum_{\ell=1}^q C_\ell  \ ,
\end{align*}
where, in the second line, we have used the fact that $\tilde{z}_0=0$. As $p$ and $q$ are relatively prime, $\sum z_{(k1)}$ must be a multiple of $p$. As each $z_n$ is in $\{0,\pm1\}$, we must have $\sum z_{(k1)}$ in $\{0,\pm p\}$.  We break this into two cases.

If $\sum z_{(k1)}=\pm p$, then $z_{(k1)}$ is constant for all $k$ with a value of $1$ or -$1$.  By our indexing and Definition~\ref{df:4}, ${\tt z}$ contains either a $p$-gon or a negative $p$-gon.  By \eqref{eq:C_ell}, for $2\le\ell\leq q$, $z_{(k\ell)}$ must form either a $p$-gon, an empty $p$-gon, or a negative $p$-gon.   As $\tilde{z}_0=0$, there must be an equivalent number of positive and negative $p$-gons, as desired.  

If $\sum z_{(k1)}=0$, we first note that, if all the $z_{(k1)}=0$, then since not all entries of $z_n$ can be 0, there must be at least one nonzero value of $C_\ell$ in \eqref{eq:C_ell}.  This yields either a positive or negative $p$-gon.  Identical reasoning ($\tilde{z}_0=0$) leads to an equivalent number of positive and negative $p$-gons in this case.  

The remaining option is for $z_{(k1)}$ to have equal number of +1 and -1's. In this case, as $z_{(k\ell)}$ can only have entries in $\{0,\pm1\}$, $C_\ell$ must be equal to 0 in Eq.~\eqref{eq:C_ell} for all $\ell$.  By our indexing, values of $k$ for which $z_{(k1)}=\pm 1$ yield $q$-gons or ``negative" $q$-gons, respectively.  As we have equal number of +1 and -1 values, we have an equivalent number of positive and negative $q$-gons, finishing the proof in the $p\neq q$ case.  

For the $p=q$ case, we apply Theorem 2.2 in~\cite{lenstra_1978_1}, which states that the sum in \eqref{z1_thm1} vanishes if and only if, for all $1 \leq j \leq p$,
\begin{align*}
0=\sum_{n=1}^{p} z_{(jn)} \zeta_{p}^n \ ,
\end{align*}
where $(jn)$ is a composite index with $n$ referring to $p$-periodic locations in the original vector ${\tt z}$.  As in the proof of Theorem 1, this implies that for each $j$, either $z_{(jn)}$ is constant for all $n$ with a value of 0 or $\pm1$.  Similar to the previous case, as not all entries $z_n$ can be 0, there must be at least one value of $j$ for which $z_{(jn)}=\pm1$.  Applying the condition $\tilde{z}_0=0$ then results in equivalent number of positive and negative $p$-gons.

For the reverse direction, we note that, if ${\tt x}$ contains pairs of $p$- or $q$-gons and empty $p$- or $q$-gons, then ${\tt y}$, which flips the full and empty $p$- or $q$-gons and agrees at all other entries with ${\tt x}$, satisfies $\tilde{x}_0=\tilde{y}_0$ and $\tilde{x}_1=\tilde{y}_1$.

\section{Proof of Lemma~\ref{lm:2}}
\label{app:B}

The case $L=1$ is a direct consequence of Lemma~\ref{lm:1}. For $L>1$, Lemma~\ref{lm:1} states that, if ${\tt y}$ is not $1$-distinguishable from ${\tt x}$, then ${\tt z}={\tt x}-{\tt y}$ has a pair of positive and negative $L$-gons.  This may not be the only such pair of $L$-gons.  However, the proof of Lemma~\ref{lm:1} showed that ${\tt z}$ cannot also have any positive  and negative $L'$-gon pairs for $L'\neq L$.  We will rewrite the DFT coefficients of ${\tt z}$ by grouping the roots of unity which form each $L$-gon together. (with all other entries being 0).  Each $L$-gon uses the roots (for some $\ell$)
\begin{align}
\label{zeta_polygon}
\{\zeta_N^{\ell+kN/L}:1\leq k\leq L\} = \{\zeta_N^\ell\zeta_L^k: 1\leq k\leq L \} \ .
\end{align}
We thus write
\begin{align}
\tilde{z}_n= &\Big{(} \sum_{\ell\in\mathcal{S}^+} \zeta_N^\ell - \sum_{\ell\in\mathcal{S}^-} \zeta_N^\ell \Big{)}
\sum_{k=1}^{L}\zeta_{L}^{nk} \ ,
\label{zeta_sum}
\end{align}
where $\mathcal{S}^+$ and $\mathcal{S}^-$ are index sets that track the appropriate positive and negative coefficients for the polygons.   As the polygons come in pairs, $| \mathcal{S}^+| = |\mathcal{S}^-|$.
We claim that 
\begin{align}
\label{zeta_sum_cases}
\sum_{k=1}^{L}\zeta_{L}^{nk}=\begin{cases}
L   \ , & n=L \\
0   \ , & 1\leq n < L
\end{cases}
\end{align}
The case $n=L$ is straightforward. As $\zeta_{L}$ is an $L$-th root of unity, $\zeta_{L}^{kL}=1$ for all $k$. For $n<L$, we can analyze the sums using a group theory approach. 

As $\zeta_{L}$ is a generator of the cyclic group of order $L$, we have $\zeta_{L}^{nk} = \zeta_{L}^{nk'}$ if and only if $nk\equiv nk'$ (mod $L$) \cite{ireland_book_1990}. This is equivalent to 
\begin{align}
\label{mod_eq}
k\equiv k' \ \big{(} {\rm mod} \frac{L}{{\rm gcd}(L,n)} \big{)} \ .
\end{align}
As $L$ must be prime (it is either $p$ or $q$) and $n<L$, ${\rm gcd}(L,n)=1$.  As $1\leq k,k'\leq L$, \eqref{mod_eq} implies that $\zeta_L^{nk} \neq\zeta_{L}^{nk'}$ for all $k$ and $k'$.  Thus, \eqref{zeta_sum_cases} is just a permuted sum of all $L$-th roots of unity, which is still 0.  

With \eqref{zeta_sum_cases} now verified, we can apply this to \eqref{zeta_sum} as
\begin{align}
\label{zn_reduced}
\tilde{z}_ n= \begin{cases}L \Big{(} \sum_{\ell\in\mathcal{S}^+} \zeta_N^\ell -\sum_{\ell\in\mathcal{S}^-} \zeta_N^\ell \Big{)} \ ,  & n=L \\
0 \ , & 1\leq n<L
\end{cases}\ .
\end{align}
All that is left to do is to verify that the sum for $\tilde{z}_L$ in \eqref{zn_reduced} is nonzero. This is a sum of roots of unity of $N$th order with equal number of coefficients of $\pm1$.  
As in the proof of Lemma~\ref{lm:1}, this term vanishes if and only if $\zeta_N^\ell \in \mathcal{S}^+$ and the $\zeta_N^\ell \in \mathcal{S}^-$ are all used to form an equivalent number of $p$- and $q$-gons.

If $\zeta_N^\ell \in \mathcal{S}^+$, then $\zeta_N^{\ell+kN/L} \notin \mathcal{S}^+$ for all $k$ as these rotations would give the same $L$-gon.  Thus by the form given by \eqref{zeta_polygon}, $\mathcal{S}^+$ cannot have any $L$-gons.  Likewise, $\mathcal{S}^+$ cannot have any $L'$-gons, as this would require $|S^+|=L'$.  This implies  there are $L'$ many $L$-gons in ${\tt z}$.  As $L'$ can only be $p$ or $q$, $LL'=N$ implying that $\tilde{x}_0=N$ which is not true.  Identical reasoning applies to $S^-$, and hence $\tilde{z}_L\neq 0$.  Therefore, ${\tt x}$ is pairwise $L$-distinguishable from all other vectors in $\Omega(N,r)$.

\begin{IEEEbiography}[{\includegraphics[width=1in,height=1.25in,clip,keepaspectratio]{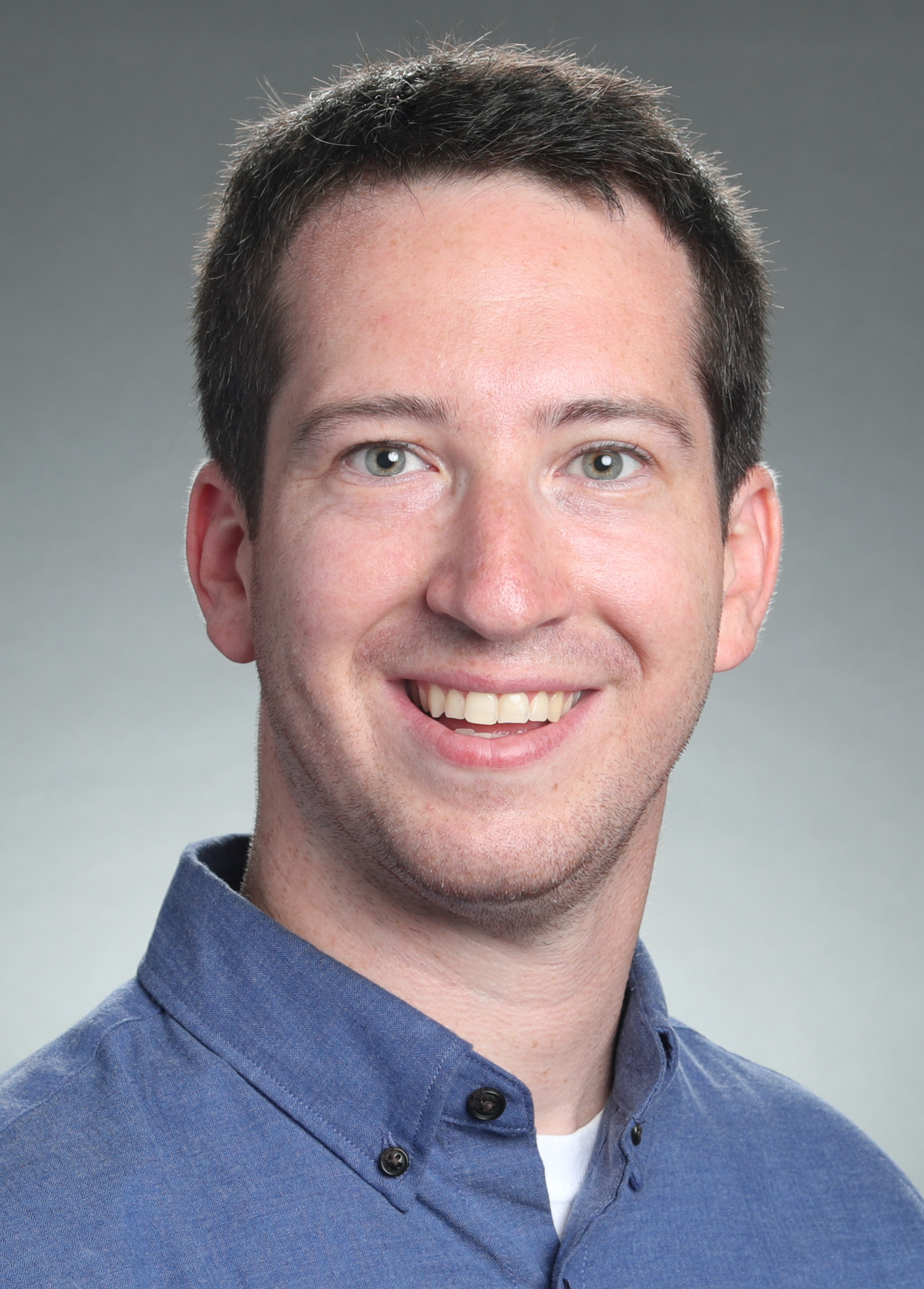}}]{Howard Levinson}
received the B.A. degree in mathematics from Tufts University in 2011, and the Ph.D. degree in Applied Mathematics and Computational Science from the University of Pennsylvania in 2016.  He was a James Van Loo Postdoctoral Assistant Professor at the University of Michigan from 2016 to 2019.  He is currently an Assistant Professor in the Department of Mathematics and Computer Science at Santa Clara University.   His current research interests include inverse problems, imaging, and fast algorithms.
\end{IEEEbiography}

\begin{IEEEbiography}[{\includegraphics[width=1in,height=1.25in,clip,keepaspectratio]{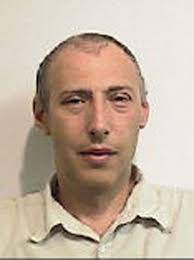}}]{Vadim Markel}
earned an undergraduate degree in Physics with specialization in Quantum Optics from Novosibirsk State University in 1987. From 1987 through 1993, he worked at the Institute of Automation and Electrometry of the Siberian Branch of the Russian Academy of Sciences conducting research in nonlinear and statistical optics. He received PhD in Physics from New Mexico State University in 1996 and continued with postdoctoral studies in quantum many-body theory at the University of Georgia (1998-1999) and in inverse problems and imaging at the Washington University–St. Louis (1999-2001). Dr. Markel has spent two years (2015-2017) as an A*MIDEX Excellence Chair at the University of Aix-Marseille and {\em Institut Fresnel} in France. He is currently an Associate Professor of Radiology and Bioengineering (secondary) and a member of Graduate Group in Applied and Computational Mathematics at the University of Pennsylvania. In he was selected as an Outstanding Referee in 2016 by the American Physical Society, and in 2018 by the Institute of Physics (UK).
\end{IEEEbiography}

\end{document}